\documentclass[10pt]{article}

\usepackage{amsmath,amsthm,amssymb}
\usepackage{mathtools}
\usepackage{hyperref}
\def\integers{{\mathbb{Z}}}
\def\reals{{\mathbb R}}
\def\naturals{{\mathbb N}}
\def\rationals{{\mathbb Q}}

\def\fancyA{\mathcal{A}}
\def\A{\mathcal{A}}

\def\D{\mathcal{D}}

\def\sups{\supseteq}
\def\subs{\subseteq}
\def\lcm{\mathrm{lcm}}

\usepackage{xcolor}
\usepackage{mathabx}

\usepackage{faktor}\usepackage{amsmath}\usepackage{amssymb}
\theoremstyle{definition}

\newtheorem{theorem}{Theorem}

\newtheorem{lemma}{Lemma}
\newtheorem{definition}{Definition}
\newtheorem{remark}{Remark}

\usepackage[shortlabels]{enumitem}
\usepackage{geometry}
\usepackage{float}
\usepackage{tikz}
\usetikzlibrary{arrows.meta}
\graphicspath{ {./Pictures/} }

\def\partialtheoremname{}
\newtheorem*{partialtheorem}{\partialtheoremname}

\usepackage[
backend=biber,
style=alphabetic,
sorting=nyt
]{biblatex}
\begin{document}
\begin{center} \Large \textbf{Commensurability and quasi-isometry classification for one vertex one loop tubular groups}
\normalsize

 \medskip Amy Tao
\end{center}

\begin{abstract}
A tubular group has a graph of groups decomposition with $\integers^2$ vertex groups and $\integers$ edge groups. This paper gives a classification of one vertex one loop tubular groups $G_{(k,\ell),(m,n)} = \langle a,b,t: [a,b]=1, t a^m b^n t^{-1} = a^k b^\ell \rangle$ up to commensurability and quasi-isometry. We show that tubular groups where the images of the edge maps have nonzero ``intersection number" are all commensurable and so quasi-isometric. When the intersection number is zero, there are two quasi-isometry classes and infinitely many commensurability classes. The nonzero and zero intersection number tubular groups are not quasi-isometric. 
\end{abstract}

\section{Introduction}

Two groups $H_1$ and $H_2$ are \textit{(abstractly) commensurable} if they contain isomorphic finite index subgroups. This is an equivalence relation on groups and especially for infinite groups, commensurability is a way of saying that two groups are ``alike" from an algebraic point of view. 

Infinite finitely generated groups also can be compared from a geometric point of view. Gromov introduced the program of studying the algebraic and geometric structure of finitely generated groups. The Cayley graph of a group with respect to a finite generating set has the word metric, under which it becomes a metric space. This metric depends on the choice of finite generating set, so we consider $G$ as a metric space up to \textit{quasi-isometric equivalence}, see Definition \ref{QI}. 

Aside from the same group being quasi-isometric to itself (with respect to different metrics), different groups can also be quasi-isometric. For example, if $H$ is a finite index subgroup of a finitely generated group $G$, then $H$ is quasi-isometric to $G$. From this, we get that if two finitely generated groups are abstractly commensurable, they are quasi-isometric. The converse is not true in general and commensurability is often stronger that quasi-isometry. In this paper, we find families of both kinds of phenomena: where commensurability occurs if and only if groups are quasi-isometric, and where it does not.

\begin{theorem}\label{thm} For  pairs of integers $(k,\ell),(m,n)\ne (0,0)$, let  $$G_{(k,\ell),(m,n)} = \langle a,b,t: [a,b]=1, t a^m b^n t^{-1} = a^k b^\ell \rangle.$$

\begin{enumerate}[label=\Alph*.]
    \item (Commensurability classification) The groups $G_{(k,\ell),(m,n)}$ where $kn - \ell m \ne 0$ are all commensurable to each other. 
For $kn - \ell m=0$, the group $G_{(k,\ell),(m,n)}$ is isomorphic to $G_{(\gcd(k,\ell),0),(\gcd(m,n),0)}$. Two groups of this form $G_{(m_1,0),(n_1,0)}, G_{(m_2,0),(n_2,0)}$ where $1\leq |m_i| \leq n_i$ are commensurable if and only if:
\begin{itemize}
    \item $m_1 = \pm m_2$, i.e.
    $G_{(m,0),(n,0)} \sim G_{(-m,0),(n,0)}$  
    \item $\frac{m_2}{m_1} = \frac{n_2}{n_1}$, i.e.
    $G_{(m,0),(n,0)} \sim G_{(cm,0), (cn,0)}$ for  $c\in \integers_{\ne 0}$.
\end{itemize} 

\item (Quasi-isometry classification) The groups $G_{(k,\ell),(m,n)}$ fall into the following three quasi-isometry classes:
\begin{itemize}
    \item $G_{(k,\ell),(m,n)}$ where $kn - \ell m \ne 0$
    \item $G_{(k,\ell),(\pm k, \pm \ell)}$. (These  are isomorphic to groups of the form $G_{(m,0),(m,0)}$) %This quasi-isometry class of tubular groups is quasi-isometric to $BS(m,m)$ for $m\geq 2$ 
    \item $G_{(k,\ell),(m,n)}$ where $kn - \ell m = 0$ but $(k,\ell)\ne (\pm m,\pm n)$. (These are isomorphic to groups of the form $G_{(m,0),(n,0)}$ for $m\ne \pm n$.) These groups are quasi-isometric to $BS(2,3)$. 
\end{itemize}

\end{enumerate}
\end{theorem}

There are often more geometric results from understanding groups in terms of quasi-isometries than algebraic results about commensurability. For example, Whyte showed that the \textit{Baumslag--Solitar groups}  $BS(m,n) = \langle a, t: ta^m t^{-1} = a^n \rangle$ with  $|m|,|n|\ne 1$ and $|m|\ne |n|$ are all quasi-isometric to each other, while those with $|m|=|n|\ne 1$ are in a separate class \cite{Whyte}. In contrast, for a long time there were only partial results about the commensurability of those Baumslag--Solitar groups. Whyte had shown that no two groups $BS(m, n)$ with $\gcd(m,n)=1$ and $m,n\ne 1$ are commensurable \cite{Whyte}, but it was  much later that Casals-Ruiz, Kazachkov and Zakharov gave a complete classification up to commensurability \cite{CRKZ}. They showed that for $BS(m,n)$ with $|m|,|n|\ne 1$, the only commensurability that occurs are swaps $BS(m,n)\sim BS(n,m)$, sign changes $BS(m,n) \sim BS(\pm m,n)$ and specific types of proportionality $BS(k,kn) \sim BS(\ell,\ell n)$ for $k,\ell,n\in \naturals$, $k,\ell >1$. Notice that for these non-solvable Baumslag--Solitar groups, the geometric and algebraic perspectives do not coincide. There are two quasi-isometry classes and infinitely many commensurability classes. 

The story is very different for the solvable Baumslag--Solitar groups. Farb and Mosher showed  that  $BS(1,n)$ and $BS(1,m)$ are quasi-isometric if and only if they are commensurable if and only if there exist $r, k , \ell$ such that $n = r^k$, $m=r^\ell$ \cite{FarbMosher}. For the  solvable Baumslag--Solitar groups, the geometric and algebraic perspectives  agree fully. 

A natural generalization of the $BS(m,n)$ are \textit{generalized Baumslag--Solitar (GBS) groups}, which are  fundamental groups of finite graphs of groups where all vertex and edge groups are $\integers$. Compared to $BS(m,n)$, the underlying graph no longer has to  be a single loop and can be any finite graph. On the geometric side, Whyte showed that all GBS groups are  quasi-isometric to one of $BS(1,n)$, $BS(n,n)$ or $BS(2,3)$ \cite{Whyte}. Meanwhile, the commensurability problem is again more difficult, and even the isomorphism problem is hard, because different graphs of $\integers$'s can describe isomorphic GBS groups, see \cite{CF08}. Clay and Forester studied commensurability for GBS groups using their actions on Bass-Serre trees \cite{CF09}. Levitt studied GBS groups without proper plateaus (a condition that generalizes $\gcd(m,n)=1$ for $BS(m,n)$) and gave some commensurability results there \cite{Lev}. In general, showing commensurability of two GBS groups often involves understanding their finite index subgroups and understanding various transformations that give isomorphic groups. Meanwhile, showing two GBS groups are incommensurable usually involves finding  commensurability invariants. Forester introduced the depth profile and used it to demonstrate incommensurable GBS groups in \cite{For24}. Casals-Ruiz, Kazachkov and Zakharov introduced the CRKZ invariant vector, which they used to distinguish the Baumslag--Solitar groups \cite{CRKZ}. Verma generalizes the CRKZ invariant vector to a larger class of GBS groups in \cite{Verma}. In many other cases, the commensurability problem for GBS groups is still open. The groups $G_{(k,\ell),(m,n)}$ with $kn-\ell m =0$ are actually GBS groups in disguise. 

Another direction of study is to increase the dimension of the  groups appearing in a graph of groups. For example, the groups $G_{(k,\ell),(m,n)}$ generalize  $BS(m,n)$ in having a one vertex and one loop underlying graph, but the vertex group is now $\integers^2$. More generally, \textit{tubular groups} are fundamental groups of graphs of groups with $\integers^2$ vertex groups and $\integers$ edge groups. Topologically, they correspond to spaces built by gluing annuli to tori. Brady and Bridson gave examples of these groups when studying the isoperimetric spectrum \cite{BradyBridson}. Right-angled Artin groups with trees as their defining graphs are tubular groups. Tubular groups are a class of groups with surprisingly diverse behavior (see \cite{Resfinite} for more of their appearances in the literature). More recently, people have studied when tubular groups  are virtually special \cite{WoodhouseSpecial}, residually finite \cite{Resfinite}, or when they act freely on a $CAT(0)$ cube complex \cite{Cat0}. There has been some progress in understanding  quasi-isometry and commensurability for tubular groups. 
Cashen studied tubular groups that satisfy the ``crossing graph condition", where at every  $\integers^2$ vertex group, the incident edge groups must rationally span the vertex group. Cashen gave an algorithm that decides in finite time  whether or not two such tubular groups are quasi-isometric \cite{Cashen}. Cashen also introduced the maximum slope invariant as a quasi-isometry invariant for tubular groups where at every vertex, the incident edge groups have images contained in three distinct cyclic subgroups \cite{CashenMaxSlope}. Casals-Ruiz, Kazachkov and Zakharov characterized the commensurability classes of tubular groups which are right-angled Artin groups defined by trees of diameter 4. As part of this, they gave the first examples of right-angled Artin groups  that are quasi-isometric but not commensurable \cite{CRKZRAAGs}. Theorem \ref{thm} gives the quasi-isometry and commensurability classification for a class of tubular groups that are not right-angled Artin groups. 
%Say something about how the ones here aren't RAAGs? 

\subsection{Outline}

In Section 2, we give background on  graphs of groups and their subgroups as well as generalized Baumslag--Solitar groups. We also include background on JSJ decompositions, which are ``maximal" graph of groups decompositions over two-ended edge groups. We then focus on one vertex one loop tubular groups.

By results in \cite{Cashen}, we already know that
the $G_{(k,\ell),(m,n)}$ with nonzero ``intersection number" $kn - \ell m$ are quasi-isometric to each other. Here, we construct finite index (normal) subgroups to demonstrate that this subfamily of one vertex one loop tubular groups  are all commensurable to $G_{(1,0),(0,1)}$ and so to each other. 

The groups $G_{(k,\ell),(m,n)}$ with $kn -\ell m = 0$ are not covered by results in \cite{Cashen}.  We show that these groups are secretly GBS groups with graph of $\integers$'s structure having one vertex and two loops. We  use tools from GBS theory like the modular homomorphism, depth profile and slide moves to distinguish infinitely many commensurability classes. Finally, we show that tubular groups with nonzero and zero intersection number are not quasi-isometric to each other. We do this by arguing that the one vertex one loop tubular group decomposition of $G_{(1,0),(0,1)}$ and the GBS decomposition of $G_{(k,\ell),(m,n)}$ with $kn -\ell m = 0$ are  JSJ decompositions. We then show that the associated trees of cylinders are not isomorphic, so that the groups themselves cannot be quasi-isometric. 

%Say something about comparisons to GBS and to BS?

\subsection*{Acknowledgments}
The author is deeply grateful to her advisor Tullia Dymarz for many helpful discussions. The author was supported by a Bung-Fung Lee Torng fellowship during the summer of 2026.

\section{Preliminaries}
\subsection{Graphs of groups}

We use the notation for graphs of groups from \cite{ScottWall}. A \textit{graph} $A$ consists of two sets $V(A)$ (the vertices of $A$) and $E(A)$ (the edges of $A$) along with an involution on $E(A)$ sending $e\in E(A)$ to $\overline{e} \in E(A)$ (with $e\ne \overline{e}$) and two maps $\partial_0, \partial_1:E(A) \to V(A)$ satisfying $\partial_1(e) = \partial_0(\overline{e})$. For an edge $e$, its initial vertex is $\partial_0(e)$, its terminal vertex is $\partial_1(e)$ and $\overline{e}$ is the reverse edge. 

A \textit{directed graph} is a graph together with a partition $E(A) = E^+(A) \sqcup E^-(A)$ that separates every pair $\{e,\overline{e}\}$. The edges in $E^+(A)$ are called \textit{directed edges}. They are the directions of edges we draw in pictures of graphs. 

A \textit{graph of groups} is a finite graph $A$ with the following data:
\begin{enumerate}
    \item Every vertex $v\in V(A)$ has an associated \textit{vertex group} $G_v$
    \item Every edge $e\in E(A)$ has an associated edge group $G_e$ such that $G_{\overline{e}}=G_e$
    \item For every $e\in E(A)$, we have an injective homomorphism $\phi_e: G_e \to G_{\partial_0(e)}$ called an \textit{edge map}.
\end{enumerate}

We construct a corresponding \textit{graph of spaces} $X$  by realizing each piece of the graph of groups topologically and then gluing them together. Namely
\begin{enumerate}
    \item For every $v\in V(A)$, we have a \textit{vertex space} $X_v$ with $\pi_1(X_v)\simeq G_v$
    \item For every $e\in E(A)$, we have a \textit{edge space} $X_e$ with $\pi_1(X_e)\simeq G_e$ and $X_{\overline{e}}=X_e$
    \item For every $e\in E(A)$, we have a map of spaces $X_e \to X_{\partial_0(e)}$ that induces the  homomorphism $G_e \to G_{\partial_0(e)}$ on the level of fundamental groups
\end{enumerate}
and then $X$ is the quotient space $$X = \left(\bigsqcup_{v\in V(A)} X_v \sqcup \bigsqcup_{e\in E(A)} X_e x [0,1] \right)/\sim$$ where the relation $\sim$ identifies $X_e \times \{0\}$ with $X_{\overline{e}} \times \{1\}$ and identifies $(x,1)\in X_e \times [0,1]$ with $\phi_e(x)$. 

The \textit{fundamental group of a graph of groups} $A$ is the fundamental group of the corresponding \textit{graph of spaces} $X$. 

\subsection{Generalized Baumslag--Solitar groups}
\textit{Generalized Baumslag--Solitar groups} (GBS groups) are fundamental groups of  graphs of groups with $\integers$ vertex and edge groups. Every edge map $\phi_e$ is an injective map $\integers\to \integers$, so it is multiplication by some nonzero integer $\lambda(e)$. 
A \textit{labelled graph} $(A,\lambda)$ is a finite graph with a label function $\lambda:E(A) \to \integers \setminus \{0\}$ describing the edge maps.  Labelled graphs are concise way of specifying GBS groups. 

We say that $(A,\lambda)$ is \textit{reduced} if every edge $e$ with $\lambda(e)=\pm 1$ is a loop. 

 \begin{definition}
     A GBS group is \textit{elementary} if it is isomorphic to $\integers$, $\integers^2=BS(1,1)$ or the Klein bottle group $K=BS(1,-1)$. Otherwise, the GBS group is called \textit{non-elementary}. 
 \end{definition}

For a GBS group given by labelled graph $(A,\lambda)$, we say an element $g\in G$ is \textit{elliptic} if it fixes a point in the Bass-Serre tree associated to $A$ (this is also called a \textit{GBS tree}). Otherwise, an element of $G$ is called \textit{hyperbolic}. If $G$ is non-elementary GBS, any two GBS trees for the same group $G$ define the same partition of $G$ into elliptic and hyperbolic elements, so we can call elements  ``elliptic" without referring to a particular GBS tree (or labelled graph). This is important in the definition of the \textit{modular homomorphism}. 

\begin{definition}
The \textit{modular homomorphism} of a non-elementary GBS group $G$ is a map $q:G \to \rationals^*$. Fix a non-trivial elliptic element $a\in G$.  Since elliptic elements in a GBS group are commensurable, every element $g\in G$ satisfies a relation $g a^k g^{-1} = a^\ell$. The map $q$ is  defined by $q(g)=\frac{k}{\ell}$. (The numbers $k$ and $\ell$ may depend on the choice of $a$, but $\frac{k}{\ell}$ will not.)

(The fact that elliptic elements in a GBS group are commensurable comes from how in a GBS tree, every edge stabilizer has finite index in its neighboring vertex stabilizers and so all edge and vertex stabilizers are in the same commensurability class.) 
\end{definition}

Notice that the modular homomorphism outputs $1$ on every elliptic element (it doesn't matter which elliptic $a\in G$ we use, so use the elliptic element itself). As such, the modular homomorphism factors through $\pi_1(A)$, the fundamental group of the underlying graph $A$. This gives us another way to compute the image of a group under the modular homomorphism, just by looking at loops in the graph $A$. 

\begin{definition}\label{modularhom} If $\gamma \in \pi_1(A)$ (the fundamental group of the underlying graph $A$) is represented by an edge-loop $(e_1,\dots,e_k)$, then $q(\gamma)=\prod_{i=1}^k \frac{\lambda(e_i)}{\lambda(\overline{e_i})}$.
    
\end{definition}

\begin{remark}(Remark 2.2 in \cite{CRKZ})
    If $H$ is a finite index subgroup of a GBS group $G$, then the modular homomorphism for $H$ is just the restriction of that for $G$. This way, $q(H)$ is a finite index subgroup of $q(G)$. 
\end{remark}

\begin{definition}
    Any two GBS trees for a non-elementary GBS group $G$ are related
by an \textit{elementary deformation}. That is, they are related by a finite sequence of elementary
moves, called elementary collapses and expansions. There are also slide moves. (See \cite{For24} for more details on elementary moves.)

\begin{center}
\includegraphics[height=1.5cm]{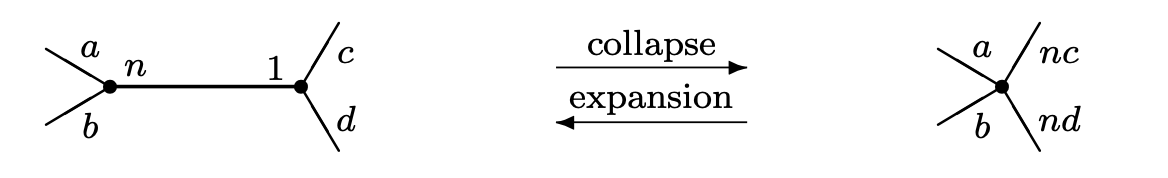}
\newline \includegraphics[height=2.5cm]{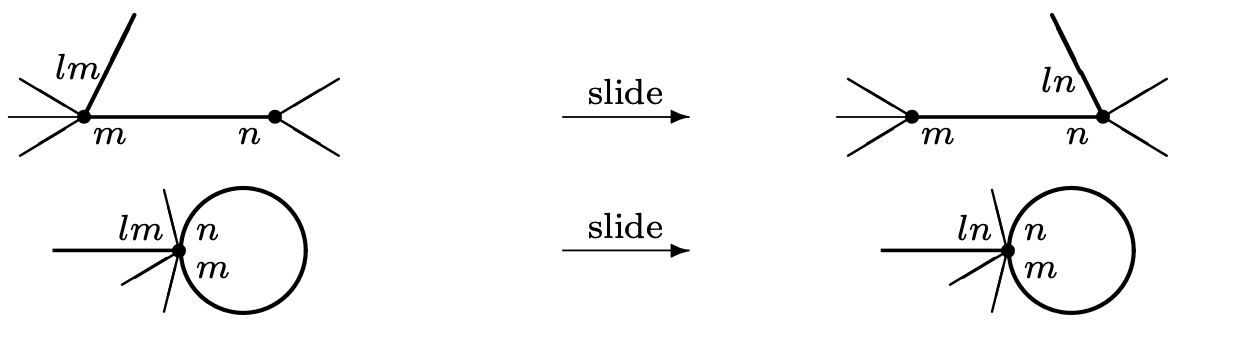}
\end{center}

\end{definition}

\begin{theorem}\label{Prop3.5For}(Theorem 7.4 in \cite{For06})
    Suppose $(A_1,\lambda_1)$ and $(A_2,\lambda_2)$ are compact reduced labelled graphs representing the same GBS group $G$. If $q(G)\cap \integers=1$, then $A_1$ and $A_2$ are related by slide moves. 
\end{theorem}

Every finite index subgroup of a GBS group is a GBS group. The inclusion of GBS groups $H \hookrightarrow G$ is induced by a covering map between their fibered 2-complexes. In \cite{For24}, Forester characterizes these covering spaces in terms of   \textit{admissible branched coverings} of their labelled graphs. 

\begin{definition}
    An \textit{admissible branched covering} of labelled graphs  $(A,\lambda) \to (B,\mu)$ is a surjective graph morphism $p: A \to B$ between connected graphs together with a degree function $$d:E(A) \sqcup V(A) \to \naturals$$ satisfying $d(e)=d(\overline{e})$ for all $e\in E(A)$ such that the following holds. Given an edge $e\in E(B)$ with $\partial_0(e)=v$ and a vertex $u\in p^{-1}(v)$, let $k_{u,e} = \gcd(d(u),\mu(e))$. Then:
    \begin{enumerate}
        \item $|p^{-1}(e) \cap E_0(u)| = k_{u,e}$
        \item $\lambda(e') = \mu(e)/k_{u,e}$ for each edge $e'\in p^{-1}(e) \cap E_0(u)$
        \item $d(e') = d(u)/k_{u,e}$ for each edge $e'\in p^{-1}(e) \cap E_0(u)$.
    \end{enumerate}
\end{definition}

\begin{theorem}[Proposition 3.19 in \cite{For24}] Let $G$ be a GBS group with labeled graph $(B,\mu)$. There is a one-to-one
correspondence between conjugacy classes of GBS subgroups of $G$ (excluding hyperbolic
cyclic subgroups) and admissible branched coverings $(A,\lambda) \to (B,\mu)$.
\end{theorem}

\begin{definition}
   A non-empty connected subgraph $P \subs A$ is a \textit{$p$-plateau} if the following condition holds for every (oriented) edge $e$ with $v=\partial_0(e)$ belonging to $P$: the label $\lambda(e)$ is divisible by $p$ if and only if $e$ is not contained in $P$. 
\end{definition}

A  nice thing about GBS groups without proper $p$-plateau (for any prime $p$) is that usually, subgroups of GBS groups are described by admissible branched covers of  labelled graphs. But if the GBS group has no proper $p$-plateau, these  come from topological covers and no branching occurs. 

\begin{theorem}\label{Lev15}(Proposition 6.5 in \cite{Lev})
Given a connected labeled graph $(A, \lambda)$, the following conditions are
equivalent:
\begin{itemize}
    \item Every admissible branched covering $p:A' \to A$ is a topological covering
    \item $A$ contains no proper plateau.
\end{itemize}
\end{theorem}

It is also useful to know the \textit{rank} of a GBS group $G$, which is the minimal cardinality of a generating set for $G$. The rank of a GBS group is an invariant of the group and notably it is computable from any labelled graph representing the group. 

\begin{theorem}[Theorem 3.2 in \cite{Lev}]
Let $G$ be a GBS group represented by a labelled graph $(A,\lambda)$. 
The rank of $G$ equals $\beta(A) + \mu(A)$, where $\beta(A)$ is the first Betti number of $A$ and $\mu(A)$ is its plateaunic number, the minimal cardinality of a set of vertices meeting
every plateau. 
\end{theorem}

We will also use the \textit{depth profile}, a commensurability invariant for GBS groups used in \cite{For24} and \cite{Verma}. \begin{definition}
    For a non-elementary GBS group $G$ and $V\leq G$ a nontrivial elliptic subgroup, the \textit{depth profile} is given by $$\D(G,V) = \{[V:V\cap gVg^{-1}]: g\in G \text{ is hyperbolic and }q(g)=\pm 1\}.$$
\end{definition}
The depth profile $\D(G,V)$ is a subset of $\naturals$. It becomes a commensurability invariant of the group $G$ when considered under the following equivalence relation on subsets of $\naturals$.  We declare that $S\subs N$ is equivalent to the set $S/r:= \{\frac{n}{\gcd(n,r)}: n\in S\}$ for each $r\in R$, and then take the symmetric and transitive closure of this relation. 

The depth profile is not so easy to calculate, but we will only need to in a case which is covered by the following theorem. 

\begin{theorem}[Proposition 6.2 in \cite{For24}]\label{Prop6.2For} Suppose $G=BS(1,N) \vee \bigvee_{i=1}^r BS(n_i, n_i)$ for some $r\geq 1$ and suppose that $N>1$, each $n_i$ divides $N$, and the set $\{n_1,\dots,n_r,N\}$ is closed under taking least common multiples and contains $1$. Let $V$ be the vertex group. Then 
$$\D(G,V) = \{N^i n_j: i\in \naturals\cup \{0\}, j=1,\dots,r\}.$$
    
\end{theorem}

For the quasi-isometric classification of the groups $G_{(m,0),(n,0)}$ later on, we will also use the following theorem. 

\begin{theorem}[Theorem 0.1 in \cite{Whyte}]\label{Whyte}
If $G$ is a GBS group, then exactly one of the following is true:
\begin{enumerate}
    \item $G$ contains a subgroup of finite index of the form $F_n \times \integers$
    \item $G=BS(1,n)$ for some $n>1$
    \item $G$ is quasi-isometric to $BS(2,3)$. 
\end{enumerate}
\end{theorem}

\subsection{Subgroups of  graphs of groups}
Let $G$ be the fundamental group of a  graph of groups  with underlying graph consisting of one vertex and one loop. 
The following theorem from Scott and Wall describes a graph of groups structure on a subgroup $H \leq G$. 

\begin{theorem} (Theorem 3.14 in \cite{ScottWall}) If $H$ is a subgroup of $G = A*_C$, then $H$ is the fundamental group of a graph of groups $\Gamma$. The vertices of $\Gamma$ correspond to double cosets $HgA$ (for $g\in G$) and the corresponding groups are $H \cap gAg^{-1}$. The edges of $\Gamma$ correspond to the double cosets $HgC$ and the corresponding groups are $H \cap gCg^{-1}$. The injections of the corresponding groups are inclusions or conjugation maps. 
\end{theorem}

\begin{remark}
 When writing down the graph of groups, we choose a set of representatives for the double cosets $HgA, HgC$. We can always choose the same representative $g$ for the edge $HgC$ as for its starting vertex $HgA$. Then the injection from the group for vertex $HgC$ to the group for $HgA$ is the inclusion map. If $HgtA=Hg'A$ for one of the chosen representatives $g'$, then there exists $h\in H$ such that $h gt \widehat{a} = g'$ for some $h\in H$ and $\widehat{a}\in A$. 
 The injection from the group for vertex $HgC$ to $HgtA$ is conjugation by $h$. 
\end{remark}

\begin{center}
\begin{tikzpicture}
     \node at (0,0) (1){$\langle a,b \rangle$};
\path[every node/.style={font=\sffamily\small}]
(1)   edge[in=22.5,out=-22.5, loop, distance=3.5cm] node[right]  {$\langle a^{m} \rangle$} (1);
\node at (2,1) (dots) {$a^{m}$};
\node at (2,-1) (dots) {$b$};
\node at (2,-2) (label) {$G_{(m,0),(0,1)}$};

\node at (6,0) (arrow) {$\longleftarrow$};

\node at (8,0) (1){$\langle a^{m},b \rangle$};
\path[every node/.style={font=\sffamily\small}]
(1)   edge[in=45,out=15, loop, distance=3.5cm] node[right]  {$\langle a^{m} \rangle$} (1);
\path (1)   edge[in=-45,out=-15, loop, distance=3.5cm] node[right]  {} (1);
\node at (11,0) (dots) {$\vdots \hspace{1cm} m$  many loops};
\node at (11,-2) (label) {Subgroup $H$ of $G_{(m,0),(0,1)}$};

\node at (9.5,1.8) (dots) {$a^{m}$};
\node at (10.7,0.8) (dots) {$b$};
 \end{tikzpicture}
 \end{center}

For example, let $H$ be the kernel of the map $G_{(m,0),(0,1)}\to \integers/m\integers$ sending $a\mapsto 1$, $b \mapsto 0$, $t \mapsto 0$. The subgroup $H$ has a graph of groups structure with one vertex and $m$ many loops, because there is just one double coset $H1A$ for the vertex group $A= \langle a,b:[a,b]=1\rangle$ and $m$ many double cosets $H1C, HaC, \dots, Ha^{m-1}C$ for the edge group $C=\langle a^m \rangle$. The edges all start at $H1A$ and end at $HtA=H1A$. The vertex group is $H\cap A = \langle a^m, b \rangle$ while the group for the edge $H a^i C$ is $H \cap a^i C a^{-i} = H \cap \langle a^m \rangle = \langle a^m \rangle$. For each edge, the origin edge map is inclusion $a^m \mapsto a^m$. The terminal edge map is conjugation by $t^{-1}\in H$. It sends $a^m \mapsto t^{-1} a^m t = b$.

\subsection{JSJ decompositions}
A group can have many different decompositions as a graph of groups. Among these, we are often interested in decompositions that are 'maximal' in  that they cannot be decomposed further (in a meaningful way) if we restrict the types of groups appearing in the graph of groups. The idea of a JSJ decomposition is that it is a maximal decomposition as a graph of groups with two-ended edge groups. The terminology comes from the study of 3-manifolds. See \cite{GL} for a survey of the theory of JSJ decompositions. 

    Let $T$ be a tree and let $G$ be a finitely generated group acting
on $T$ by isometries, without edge inversions. Let $\fancyA$ be the family of two-ended subgroups of $G$. We say that $T$ is an \textit{$\fancyA$-tree} if  the edge stabilizers $G_e$ of $T$ are all in $\fancyA$. (More generally, one could replace $\A$ with any family of subgroups of $G$ that is stable under taking subgroups and under conjugation and consider other types of maximal decompositions.)

\begin{definition}

A subgroup $H\leq G$ is \textit{universally elliptic} if it fixes a point in every $\A$-tree. An $\A$-tree is \textit{universally elliptic} if all of its edge stabilizers $G_e$ are universally elliptic subgroups of $G$. An $\A$-tree $T$ \textit{dominates} another $\A$-tree $T'$ if every vertex stabilizer of $T$ is elliptic in $T'$.  A \textit{JSJ tree} of $G$ is an $\A$-tree that is universally elliptic and dominates any other universally elliptic $\A$-tree $T'$. The quotient graph of groups is called a \textit{JSJ decomposition} of $G$. 

\end{definition}

JSJ trees are surveyed in \cite{GL}. A finitely presented group always has a JSJ decomposition over $\A$. However, JSJ decompositions are not unique. Instead there is a deformation space worth of JSJ decompositions, where any one can be transformed into any other through a finite sequence of certain kinds of moves (see \cite{ForDeform} and \cite{GL2}). But there is a canonical object called the \textit{JSJ tree of cylinders}. 

\begin{definition}
    Given a JSJ decomposition of $G$, let $T$ be the corresponding Bass-Serre tree. We say that edge stabilizers $G_e$ and $G_{e'}$ are \textit{commensurable} if their intersection is finite index in both $G_e$ and $G_{e'}$. This defines an equivalence relation on the edges of $T$. The equivalence classes of edges are called \textit{cylinders}. 
\end{definition}

Every cylinder is a subtree of $T$, and two distinct cylinders meet in at most one point \cite{GLCylinder}.

\begin{definition}
The \textit{tree of cylinders} of $T$ is the bipartite tree $T_c$ with vertex
set $V(T_c) = V_0(T_c) \sqcup V_1(T_c)$ and edge set $E(T_c)$ defined as follows:
\begin{itemize}
    \item $V_0(T_c)$ is the set of vertices $x$ of $T$ belonging to (at least) two distinct cylinders
    \item $V_1(T_c)$ is the set of cylinders $Y$ of $T$
    \item There is an edge between $x\in V_0(T_c)$ and $Y\in V_1(T_c)$ if and only if $x$ (viewed as a vertex of $T$) belongs to $Y$ (viewed as a subtree of $T$). 
\end{itemize}
\end{definition}

Up to $G$-equivariant isomorphism, there is a unique JSJ tree of cylinders \cite{GLCylinder}. 

\begin{definition}\label{QI}
    A map $f:(X,d_X) \to (Y,d_Y)$ is \textit{quasi-isometry} if there exist $K\geq 1$, $C\geq 0$ such that $$\frac{1}{K} d_X(x_1,x_2) - C \leq d_Y(f(x_1),f(x_2)) \leq K d_X(x_1,x_2) + C$$
for all $x_1,x_2\in X$. Quasi-isometry is an equivalence relation on metric spaces. 
\end{definition}

In general, JSJ decompositions are robust under quasi-isometries (\cite{Papasoglu}, \cite{CashenMartin}, \cite{ShepherdWoodhouse}), but the JSJ tree of cylinders can capture even more detailed information about the  decomposition.

\begin{theorem}(Theorem 2.8 in \cite{ShepherdWoodhouse})
Let $\psi:G \to G'$ be a quasi-isometry of finitely presented one-ended groups. Let $T$ and $T'$ be Bass-Serre trees  for JSJ decompositions of $G$ and $G'$ respectively. Let $T_c$ and $T_c'$ be the corresponding JSJ trees of cylinders. Then there is a unique isomorphism $\widehat{\psi}:T_c \to T_c'$ such that:
\begin{enumerate}
    \item $\widehat{\psi}(V_0 T_c) = V_0 T_c'$ and $\widehat{\psi}(V_1 T_c) = V_1 T_c'$
    \item $\psi(G_v)$ is finite Hausdorff distance from $G'_{\widehat{\psi}(v)}$ for $v\in V(T_c)$ and $\psi(G_e)$ is finite Hausdorff distance from $G'_{\widehat{\psi}(e)}$ for $e\in E(T_c)$. Moreover the restrictions $\psi:G_v \to G'_{\widehat{\psi}(v)}$ and $\psi:G_e \to G'_{\widehat{\psi}(e)}$ are quasi-isometries with respect to the intrinsic metrics of the vertex and edge stabilizers
    \item (There is one more condition about a uniform bound on images of cosets, but we will not use it.)
\end{enumerate}
    
\end{theorem}

We are interested in JSJ decompositions of certain tubular groups over two-ended edge groups. In our setting, there is an equivalent characterization of JSJ decompositions that is easier to show than than the original definition, because it works with a particular graph of groups. In order to state the equivalent characterization, we need some terminology from coarse geometry. We use this terminology minimally and include it here just for completeness.  

We say that subsets of a metric space $X$ are \textit{coarsely equivalent} if they are bounded Hausdorff distance apart. Quasi-isometries respect coarse equivalence of subsets. If $\mathcal{P}$ is a set of coarse equivalence classes of subsets of $X$ and $\mathcal{P}'$ is a set of coarse equivalence classes of subsets of another metric space $Y$, a quasi-isometry of pairs $(X, \mathcal{P}) \to (Y, \mathcal{P}')$ is a quasi-isometry $X \to Y$ that induces a bijection between $\mathcal{P}$ and $\mathcal{P}'.$

In a geodesic metric space $(X, d_X)$, a path-connected subset $L$ has the induced length metric $d_L$. A \textit{quasi-line} in $X$ is such a path-connected subset $L$ where $(L, d_L)$ is quasi-isometric to $\reals$ and there exist unbounded, non-decreasing real functions $\rho_0$ and $\rho_1$ such that for all $x, x'\in L$ we have $$\rho_0(d_L(x,x')) \leq d_X(x, x') \leq \rho_1(d_L(x,x')).$$ A \textit{peripheral structure $\mathcal{P}$} on a geodesic metric space is a collection of coarse equivalence classes of quasi-lines. In particular, if $G$ is a finitely generated group and $\mathcal{H}$ is a finite collection of two-ended subgroups of $G$, then $\mathcal{H}$ induces a peripheral structure consisting of distinct coarse equivalence classes of conjugates of elements of $\mathcal{H}$. 

Let $\Gamma$ be a graph of groups over two-ended edge groups. For a vertex group $G_v$ of $\Gamma$, the \textit{peripheral structure coming from incident edge groups}, denoted $\mathcal{P}$, is the set of  coarse equivalence classes in $G_v$ of $G_v$-conjugates of the images of the edge maps $G_{e} \to G_v$ for edges $e$ incident to $v$.

\begin{definition}
A vertex $v$ is \textit{hanging} if the pair $(G_v, \mathcal{P}_v)$ is quasi-isometric to $(\widetilde{\Sigma}, \mathcal{P}_{\partial \Sigma})$, where $\widetilde{\Sigma}$ is the universal cover of the hyperbolic pair of pants $\Sigma$, and $\mathcal{P}_{\partial \sigma}$ is the peripheral structure on the universal cover $\widetilde{\Sigma}$ consisting of the coarse equivalence classes of the components of the preimages of the boundary curves of $\Sigma$. 

A vertex $v$ is \textit{rigid} if it is not two-ended, not hanging, and does not split over a two-ended subgroup relative to its incident edge groups, i.e. in such a way that the incident edge groups are conjugate into one of the vertex groups in the supposed splitting of $G_v$. 

The notions of rigid and hanging vertices in a graph of groups extend to vertices in the tree of cylinders. 

\end{definition}

The following  equivalent characterization of JSJ decompositions  is introduced in \cite{CashenMartin}. It is used in \cite{Edletzberger} to study right-angled Coxeter groups. 

\begin{definition}\label{JSJCashen}(Definition 2.7 in \cite{CashenMartin}, or Lemma 2.15 in \cite{Edletzberger})
Let $G$ be a finitely presented one-ended group that is not commensurable to a surface group. A \textit{JSJ decomposition of $G$} is a (possibly trivial) graph of groups decomposition $\Gamma$ with two-ended edge groups satisfying the following conditions:
\begin{enumerate}
    \item Every vertex group is either two-ended, hanging or rigid
    \item If $v$ is a valence one vertex with two-ended vertex group, then the incident edge group does not surject onto $G_v$
    \item Every cylinder in the Bass-Serre tree of $\Gamma$ that contains exactly two hanging vertices also contains a rigid vertex. 
\end{enumerate}
\end{definition}

\section{Commensurability classification}
We are interested in one vertex one loop tubular groups $G_{(x,y),(z,w)}$. 
\begin{center}
\begin{tikzpicture}
\node at (0,0) (3){$\langle a,b: [a,b]=1 \rangle$};
\path[every node/.style={font=\sffamily\small}]
(3)   edge[loop] node[above]  {$\langle a^x b^y \rangle$} (3);
\node at (-1.25,0.75) {$a^z b^w$};
\node at (1.25,0.75) {$a^x b^y$};
\end{tikzpicture}
\end{center}

\subsection{Nonzero intersection number}
Here we have $G_{(x,y),(z,w)}$ where the ``intersection number" $zw-yz \ne 0$.This terminology in used \cite{Cat0}. 

\begin{lemma}\label{basis}
    We can reduce to the case where the tubular group is of the form $G_{(k,0),(j,\ell)}$. 
\end{lemma}
Starting with $G_{(x,y),(z,w)}$, write $(x,y) = k(m,n)$ where $k=\gcd(x,y)$ and $m=\frac{x}{\gcd(x,y)}, n=\frac{y}{\gcd(x,y)}$ so that $\gcd(m,n)=1$. By Bezout's lemma, there exist $M, N$ such that $mM - nN = \gcd(m,-n)=1$. Then note $\alpha= a^m b^n$ and $\beta = a^N b^M$ form another basis for the $\integers^2$ vertex group. We have $a^x b^y = (a^m b^n)^k$ and $a^z b^w = (a^m b^n)^j (a^N b^M)^\ell$, so changing to this basis, the group looks like $G_{(k,0),(j,\ell)}$. 

\begin{center}
\begin{tikzpicture}
\node at (0,0) (3){$\langle \alpha,\beta \rangle$};
\path[every node/.style={font=\sffamily\small}]
(3)   edge[loop] node[above]  {$\langle \alpha^k \rangle$} (3);
\node at (-1,0.75) {$\alpha^k$};
\node at (1,0.75) {$\alpha^j \beta^\ell$};
\end{tikzpicture}
\end{center}

The nonzero intersection number condition now says that $k \ell \ne 0$. 
% It's a fact that $(\alpha,\beta), (\delta,\gamma)$ are a $\integers$-basis for $\integers^2$ if and only if the ``determinant" $\alpha \gamma - \beta \delta = \pm 1$. 

\begin{lemma}
     $G_{(k,0),(j,\ell)}$ with nonzero intersection number is commensurable to $G_{(M,0),(0,1)}$ for some $M$. 
\end{lemma}
This takes some casework, which will be structured as:
\begin{enumerate}[label=Case \arabic*:]
    \item $\gcd(k,\ell)=1$
    \begin{enumerate}[(a)]
        \item $|k|=|\ell|=1$
        \item At least one of $|k|,|\ell|$  is greater than 1
    \end{enumerate}
    \item $\gcd(k,\ell)>1$
    \begin{enumerate}[(a)]
        \item $j=k$
        \item $|j-k|=1$
        \item $|j-k|\geq 2$
    \end{enumerate}
\end{enumerate}
In every case, we'll take surjective maps from the tubular group to a finite abelian group to get a finite index subgroup of the tubular group as kernel of the map. Taking the codomain group to be abelian just makes things easier, since the images of the vertex group generators will automatically commute. 

\begin{proof} \underline{For Case 1, we're assuming that $\gcd(k,\ell)=1$.}

\begin{enumerate}[(a)]
    \item We first consider the sub-case where $|k|=|\ell|=1$. This is actually quite simple, because $G_{(1,0),(j,1)}$ is isomorphic to $G_{(1,0),(0,1)}$ after a change of basis in the vertex group. 

\begin{center}
 \begin{tikzpicture}
     \node at (0,0) (1){$\langle a,b \rangle$};
     \node at (-1.5,0) (2){$\langle a^{\pm k},a^j b^{\pm \ell} \rangle=$};
\path[every node/.style={font=\sffamily\small}]
(1)   edge[in=22.5,out=-22.5, loop, distance=3.5cm] node[right]  {$\langle a\rangle$} (1);
\node at (2,1) (dots) {$a$};
\node at (2,-1) (dots) {$a^j b$};
\node at (1.5,-1.5) (label) {$G_{(1,0),(j,1)}$};

\node at (5.5,0) (sim) {$\simeq$};

     \node at (7,0) (1){$\langle \alpha,\beta \rangle$};
\path[every node/.style={font=\sffamily\small}]
(1)   edge[in=22.5,out=-22.5, loop, distance=3.5cm] node[right]  {$\langle \alpha\rangle$} (1);
\node at (9,1) (dots) {$\alpha$};
\node at (9,-1) (dots) {$\beta$};
\node at (9,-1.5) (label) {$G_{(1,0),(0,1)}$};

 \end{tikzpicture}
 \end{center}

\item Now suppose at least one of $|k|,|\ell|$ is strictly greater than $1$. Let $G_{(k,0),(j,\ell)} \twoheadrightarrow \integers/k\integers \oplus \integers/\ell\integers$ sending $a \mapsto (1,0), b \mapsto (c,1), t \mapsto (0,0)$ where $c$ satisfies $c\ell \equiv -j \mod k$, so that $j+c\ell \equiv 0 \mod k$ and the map is well-defined. Such a $c$ exists because $\gcd(k,\ell)=1$, so we're able to divide by $\ell$. If $k=1$, just take the map to $\integers/\ell \integers$, and similarly if $\ell=1$, just take the map to $\integers/k\integers$. What follows will work for them too. 

The kernel $K$ has graph of groups structure with one vertex (the original vertex group surjects) and $|k\ell|$ many loops (the original edge group maps to the trivial group, which has index $|k\ell|$). The edge groups and edge maps for all the loops are identical ($K$ is normal), with the edge group generator $a^k$ being sent to $a^k$ and $a^j b^\ell$.

%\tikzset{every loop/.style={min distance=30mm}}
\begin{tikzpicture}
\node at (0,0) (1){$\langle a^k,a^j b^\ell \rangle$};
\path[every node/.style={font=\sffamily\small}]
(1)   edge[in=45,out=15, loop, distance=4.5cm] node[right]  {$\langle a^k \rangle$} (1);
\path (1)   edge[in=-45,out=-15, loop, distance=4.5cm] node[right]  {} (1);
\node at (3,0) (dots) {$\vdots \hspace{1cm} |k\ell|$  many loops};
\node at (2,2) (dots) {$a^k$};
\node at (3,0.8) (dots) {$a^j b^\ell$};
\node at (2,-3) (label) {Index $|k\ell|$ subgroup K of $G_{(k,0),(j,\ell)}$};

\node at (6,0) (sim) {$\simeq$};

\node at (8,0) (1){$\langle \alpha,\beta \rangle$};
\path[every node/.style={font=\sffamily\small}]
(1)   edge[in=45,out=15, loop, distance=4.5cm] node[right]  {$\langle \alpha \rangle$} (1);
\path (1)   edge[in=-45,out=-15, loop, distance=4.5cm] node[right]  {} (1);
\node at (11,0) (dots) {$\vdots \hspace{1cm} |k\ell|$  many loops};
\node at (10,2) (dots) {$\alpha$};
\node at (11,0.8) (dots) {$\beta$};
\end{tikzpicture}

 Letting $\alpha = a^k, \beta = a^j b^\ell$, we notice that $K$'s graph of groups is the same as that of a subgroup of $G_{(k\ell,0),(0,1)}$. Namely, the kernel $H$ of the map $G_{(k\ell,0),(0,1)} \twoheadrightarrow \integers/k\ell \integers$ sending $a \mapsto 1, b \mapsto 0, t \mapsto 0$ has a graph of groups that looks the same, after renaming vertex group generators. 

 \begin{tikzpicture}
     \node at (0,0) (1){$\langle a,b \rangle$};
\path[every node/.style={font=\sffamily\small}]
(1)   edge[in=22.5,out=-22.5, loop, distance=3.5cm] node[right]  {$\langle a^{k\ell} \rangle$} (1);
\node at (2.2,1) (dots) {$a^{k\ell}$};
\node at (2.2,-0.9) (dots) {$b$};
\node at (2,-2) (label) {$G_{(k\ell,0),(0,1)}$};

\node at (6,0) (arrow) {$\longleftarrow$};

\node at (8,0) (1){$\langle a^{k\ell},b \rangle$};
\path[every node/.style={font=\sffamily\small}]
(1)   edge[in=45,out=15, loop, distance=3.5cm] node[right]  {$\langle a^{k\ell} \rangle$} (1);
\path (1)   edge[in=-45,out=-15, loop, distance=3.5cm] node[right]  {} (1);
\node at (11,0) (dots) {$\vdots \hspace{1cm} k\ell$  many loops};
\node at (11,-2) (label) {Index $k\ell$ subgroup $H$ of $G_{(k\ell,0),(0,1)}$};

\node at (10,2) (dots) {$a^{k\ell}$};
\node at (10.8,0.8) (dots) {$b$};
 \end{tikzpicture}

Therefore $G_{(k,0),(j,\ell)} \sim G_{(k\ell,0),(0,1)}$ in Case 1. 
\end{enumerate}

\bigskip \underline{For Case 2, we're assuming that $\gcd(k,\ell)>1$.}

\begin{enumerate}[(a)]
    \item Suppose that $j=k$. We proceed similarly to in Case 1 by considering $G_{(k,0),(k,\ell)} \twoheadrightarrow \integers/k\integers \oplus \integers/\ell\integers$ sending $a \mapsto (1,0), b \mapsto (0,1), t \mapsto (0,0)$. The kernel of this map has graph of groups structure with one vertex and $|k\ell|$ many loops. The vertex group is $\langle a^k, b^\ell \rangle$, but we can change the basis to  $\langle a^k, a^k b^\ell \rangle$. 
    
\begin{center}
\begin{tikzpicture}
\node at (0,0) (1){$\langle  a^k, a^k b^\ell \rangle$};
\path[every node/.style={font=\sffamily\small}]
(1)   edge[in=45,out=15, loop, distance=2.5cm] node[right]  {$\langle a^k \rangle$} (1);
\path (1)   edge[in=-45,out=-15, loop, distance=2.5cm] node[right]  {} (1);
\node at (3,0) (dots) {$\vdots \hspace{1cm} |k\ell|$  many loops};
\node at (3,0) (dots) {$\vdots \hspace{1cm} |k\ell|$  many loops};
\node at (1.5,1.5) (dots) {$a^k$};
\node at (2.2,0.5) (dots) {$a^k b^\ell$};
\node at (2,-2) (label) {Index $|k\ell|$ subgroup of $G_{(k,0),(j,\ell)}$};

\node at (5,0) (sim) {$\simeq$};

\node at (7,0) (1){$\langle  \alpha, \beta \rangle$};
\path[every node/.style={font=\sffamily\small}]
(1)   edge[in=45,out=15, loop, distance=2.5cm] node[right]  {$\langle \alpha \rangle$} (1);
\path (1)   edge[in=-45,out=-15, loop, distance=2.5cm] node[right]  {} (1);
\node at (10,0) (dots) {$\vdots \hspace{1cm} |k\ell|$  many loops};
\node at (10,0) (dots) {$\vdots \hspace{1cm} |k\ell|$  many loops};
\node at (8.5,1.5) (dots) {$\alpha$};
\node at (9.2,0.5) (dots) {$\beta$};
\node at (9,-2) (label) {Index $|k\ell|$ subgroup $H$ of $G_{(k\ell,0),(0,1)}$};
\end{tikzpicture}
\end{center}

Letting $\alpha=a^k, \beta = a^k b^\ell$, we recognize the same  subgroup $H$ of $G_{(k\ell,0),(0,1)}$ as in Case 1. So we get the commensurability we want. 

\item Suppose that $|j-k|=1$. Consider the map $G_{(k,0),(j,\ell)} \twoheadrightarrow \integers/k\ell\integers$ sending $a \mapsto \ell, b \mapsto k-j = \pm 1$ (in this sub-case, so the map does surject) and $t\mapsto 0$. The kernel of this map will have the graph of groups structure depicted below (the left side picture is for $j=k+1$ and the right side picture is for $j=k-1$. They both have $k\ell$ many loops). 

\begin{tikzpicture}
\node at (0,0) (1)[left]{$\langle  a^k, a b^\ell \rangle$};
\node at (-1.3,0) (2)[left]{$\langle a^k, a^j b^\ell\rangle$=};
\path[every node/.style={font=\sffamily\small}]
(1)   edge[in=45,out=15, loop, distance=2.5cm] node[right]  {$\langle a^k \rangle$} (1);
\path (1)   edge[in=-45,out=-15, loop, distance=2.5cm] node[right]  {} (1);
\node at (3,0) (dots) {$\vdots \hspace{0.5cm}  \: k\ell$  many loops};
\node at (1,1.5) (dots) {$a^k$};
\node at (1.4,0.5) (dots) {$a^j b^\ell$};
\node at (1,-2) (label) {For $j=k+1$ (so $k-j=-1$)};

\node at (9,0) (1)[left]{$\langle  a^k, a^{-1} b^\ell \rangle$};
\node at (7.3,0) (2)[left]{$\langle a^k, a^j b^\ell\rangle$=};
\path[every node/.style={font=\sffamily\small}]
(1)   edge[in=45,out=15, loop, distance=2.5cm] node[right]  {$\langle a^k \rangle$} (1);
\path (1)   edge[in=-45,out=-15, loop, distance=2.5cm] node[right]  {} (1);
\node at (10,1.5) (dots) {$a^k$};
\node at (10.4,0.5) (dots) {$a^j b^\ell$};
\node at (8,-2) (label) {For $j=k-1$ (so $k-j = 1$)};
\node at (11,0) (dots) {$\vdots$};
\end{tikzpicture}
In either case, we are able to modify the vertex group basis (since $j = k\pm 1$) so that renaming $\alpha = a^k, \beta = a^j b^\ell$, we recognize the same (index $k\ell$) subgroup $H \leq G_{(k\ell,0),(0,1)}$ as before. So we get the commensurability we want. 

\item Now suppose that $|j-k|\geq 2$. Consider the map $G_{(k,0),(j,\ell)} \twoheadrightarrow \integers/|j-k|\integers \: \oplus \integers/\ell\integers$ sending $a \mapsto (1,0), b\mapsto (0,1)$ and $t \mapsto (0,0)$. Since $k\equiv j \mod|j-k|$, this map is well-defined. The kernel of this map has a graph of groups structure with one vertex and $|\ell| \gcd(k,|j-k|)$ many loops (that is the index of $\langle (k,0)\rangle$ in the codomain). The loops have identical edge groups, generated by $a^{\lcm(k,|j-k|)} = (a^k)^{\frac{|j-k|}{\gcd(k,|j-k|)}}$ and edge maps are inclusion and $(a^k)^{\frac{|j-k|}{\gcd(k,|j-k|)}} \mapsto (a^j b^\ell)^{\frac{|j-k|}{\gcd(k,|j-k|)}}$. We express these in terms of the vertex group generators, depicted below. 

\begin{tikzpicture}
\node at (0,0) (1)[left]{$\langle  a^{|j-k|}, b^\ell \rangle$};
\path[every node/.style={font=\sffamily\small}]
(1)   edge[in=45,out=15, loop, distance=2.5cm] node[right]  {$\langle (a^{|j-k|})^{\frac{k}{\gcd(k,|j-k|)}} \rangle$} (1);
\path (1)   edge[in=-45,out=-15, loop, distance=2.5cm] node[right]  {} (1);
\node at (2,-0.5) (dots) {$\vdots$};
\node at (0.2,1.5) (out) {$(a^{|j-k|})^{\frac{k}{\gcd(k,|j-k|)}} $};
\node at (3.3,0.5) (in) {$(a^{|j-k|})^{\frac{j}{\gcd(k,|j-k|)}} (b^\ell)^\frac{|j-k|}{\gcd(k,|j-k|)}$};
\node at (1,-2) (label) {Index $\ell |j-k|$ subgroup $H$ of $G_{(k,0),(j,\ell)}$};

\node at (5,-0.2) (sim) {$\simeq$};

\node at (8,0) (1)[left]{$\langle  \alpha, \beta \rangle$};
\path[every node/.style={font=\sffamily\small}]
(1)   edge[in=45,out=15, loop, distance=2.5cm] node[right]  {$\langle \alpha^{\frac{k}{\gcd(k,|j-k|)}} \rangle$} (1);
\path (1)   edge[in=-45,out=-15, loop, distance=2.5cm] node[right]  {} (1);
\node at (10,-0.5) (dots) {$\vdots$};
\node at (8.2,1.5) (out) {$\alpha^{\frac{k}{\gcd(k,|j-k|)}} $};
\node at (11,0.5) (in) {$\alpha^{\frac{j}{\gcd(k,|j-k|)}} \beta^\frac{|j-k|}{\gcd(k,|j-k|)}$};
\node at (9,-2) (label) {After renaming vertex group generators};

\end{tikzpicture}

Notably now we have that $\gcd(\frac{k}{\gcd(k,|j-k|)}, \frac{|j-k|}{\gcd(k,|j-k|)})=1$. For clarity, let $m=\frac{k}{\gcd(k,|j-k|)}, p=\frac{|j-k|}{\gcd(k,|j-k|)}$ and $n=\frac{j}{\gcd(k,|j-k|)}$. Observe that $n$ is actually an integer, because $j$ is in fact divisible by $\gcd(k, |j-k|)$. 

\begin{center}
\begin{tikzpicture}
\node at (0,0) (1)[left]{$\langle  \alpha, \beta \rangle$};
\path[every node/.style={font=\sffamily\small}]
(1)   edge[in=45,out=15, loop, distance=2.5cm] node[right]  {$\langle \alpha^m \rangle$} (1);
\path (1)   edge[in=-45,out=-15, loop, distance=2.5cm] node[right]  {} (1);
\node at (5,0) (dots) {$\vdots \hspace{1cm} |\ell| \gcd(k,|j-k|)$ many  loops};
\node at (0.5,1.2) (out) {$\alpha^m$};
\node at (1.5,0.5) (in) {$\alpha^n \beta^p$};
\node at (1,-2) (label) {Still that finite index  subgroup $H$ of $G_{(k,0),(j,\ell)}$};
\end{tikzpicture}
\end{center}

Now note that  each loop  is in the situation of Case 1, because $\gcd(m,p)=1$. As such, we can do what we did in Case 1. Consider a map $H \twoheadrightarrow \integers/m\integers \oplus \integers/p\integers$ sending all of the HNN letters for the loops to zero, $\alpha \mapsto (1,0)$ and $\beta \mapsto (c,1)$ where $n+cp \equiv 0 \mod m$, so that the map is well-defined. Again, $c$ exists because we are able to divide by $p$ on both sides of $c p \equiv -n \mod m$, now that $\gcd(m,p)=1$. This produces a finite index subgroup $H'$ of $H$. It has graph of groups structure with one vertex and  $|\ell| \gcd(k,|j-k|)|mp|$ many loops (since each loop in $H$ gives rise to $|mp|$ many loops in $H'$). 

\begin{tikzpicture}
\node at (0,0) (1)[left]{$\langle  \alpha^m, \alpha^n \beta^p \rangle$};
\path[every node/.style={font=\sffamily\small}]
(1)   edge[in=45,out=15, loop, distance=2.5cm] node[right]  {$\langle \alpha^m \rangle$} (1);
\path (1)   edge[in=-45,out=-15, loop, distance=2.5cm] node[right]  {} (1);
\node at (2,0) (dots) {$\vdots$};
\node at (0.6,1.4) (out) {$\alpha^m$};
\node at (1.5,0.5) (in) {$\alpha^n \beta^p$};
\node at (1,-2) (label) {Finite index  subgroup $H'$ of $H$};
\node at (3.5,0) (sim) {$\simeq$};

\node at (7,0) (1)[left]{$\langle  \delta,\gamma \rangle$};
\path[every node/.style={font=\sffamily\small}]
(1)   edge[in=45,out=15, loop, distance=2.5cm] node[right]  {$\langle \delta \rangle$} (1);
\path (1)   edge[in=-45,out=-15, loop, distance=2.5cm] node[right]  {} (1);
\node at (10.5,0) (dots) {$\vdots$ \hspace{0.5cm} $|\ell| \gcd(k,|j-k|)|mp|$};
 \node at (11,-0.6) (dots) {many loops};
\node at (7.6,1.4) (out) {$\delta$};
\node at (8.5,0.5) (in) {$\gamma$};
\node at (7,-2) (label) {Renaming $\delta =\alpha^m, \gamma =\alpha^n \beta^p$};
\end{tikzpicture}

Renaming the vertex generators, we notice that the right hand side picture is the same as that of the index $M=|\ell| \gcd(k,|j-k|)|mp|$ subgroup of $G_{(M,0),(0,1)}$, obtained as the kernel of the map $G_{(M,0),(0,1)} \twoheadrightarrow \integers/M\integers$ sending $a \mapsto 1, b\mapsto 0, t\mapsto 0$. Since $H'$ was finite index in $H$, which was finite index in $G_{(k,0),(j,\ell)}$, this shows that $G_{(k,0),(j,\ell)}$ is commensurable to $G_{(M,0),(0,1)}$. 
\end{enumerate}

We are now done with all the cases. 
\end{proof}

\begin{lemma}
    $G_{(1,0),(0,1)}$ is commensurable to $G_{(1,0),(1,M)}$ for all $M\geq 1$, which is in turn commensurable to $G_{(M,0),(0,1)}$. From this, we get that  the one vertex one loop tubular groups with nonzero intersection number are all commensurable to each other. 
\end{lemma}

\begin{proof} 
We have already handled the case where $M=1$, as part of Case 1a), so let $M\geq 2$. Consider the map $G_{(1,0),(0,1)} \twoheadrightarrow \integers/M\integers$ sending $a \mapsto 1, b \mapsto 1, t \mapsto 0$. The kernel of this map has graph of groups structure with one vertex and one loop (because the image of the original edge group maps is already all of $\integers/M\integers$). 

\begin{center}
 \begin{tikzpicture}
     \node at (0,0) (1){$\langle a^M,a^{-1}b \rangle$};
\path[every node/.style={font=\sffamily\small}]
(1)   edge[in=22.5,out=-22.5, loop, distance=4.5cm] node[right]  {$\langle a^M \rangle$} (1);
\node at (2.5,1) (dots) {$a^M$};
\node at (2.5,-1) (dots) {$b^M = (a^M)^1 (a^{-1}b)^M$};
\node at (2,-2) (label) {Index $M$ subgroup of $G_{(1,0),(0,1)}$};

\node at (5.5,0) (sim) {$\simeq$};

     \node at (7,0) (1){$\langle \alpha,\beta \rangle$};
\path[every node/.style={font=\sffamily\small}]
(1)   edge[in=22.5,out=-22.5, loop, distance=4.5cm] node[right]  {$\langle \alpha\rangle$} (1);
\node at (9.5,1) (dots) {$\alpha$};
\node at (9.5,-1) (dots) {$\alpha \beta^M$};
\node at (9,-2) (label) {$G_{(1,0),(1,M)}$};
 \end{tikzpicture}
 \end{center}

 Letting $\alpha=a^M, \beta=a^{-1}b$, we see that $G_{(1,0),(0,1)}$ is commensurable to $G_{(1,0),(1,M)}$.

For the second part of the Lemma, consider the map $G_{(1,0),(1,M)} \twoheadrightarrow \integers/M \integers$ sending $a \mapsto 0, b \mapsto 1, t \mapsto 0$. The kernel of this map has graph of groups structure with one vertex and $k$ loops as depicted below. 

\begin{tikzpicture}
\node at (0,0) (1)[left]{$\langle  a, b^M \rangle$};
\node at (-2.2,0) (2){$\langle a,ab^M \rangle =$};
\path[every node/.style={font=\sffamily\small}]
(1)   edge[in=45,out=15, loop, distance=2.5cm] node[right]  {$\langle a \rangle$} (1);
\path (1)   edge[in=-45,out=-15, loop, distance=2.5cm] node[right]  {} (1);
\node at (2,0) (dots) {$\vdots \hspace{0.5cm} M$ loops};
\node at (0.6,1.4) (out) {$a$};
\node at (1.5,0.5) (in) {$a b^M$};
\node at (1,-2) (label) {Index $M$ subgroup of $G_{(1,0),(1,M)}$};
\node at (4,0) (sim) {$\simeq$};

\node at (7,0) (1)[left]{$\langle  \alpha,\beta \rangle$};
\path[every node/.style={font=\sffamily\small}]
(1)   edge[in=45,out=15, loop, distance=2.5cm] node[right]  {$\langle \alpha \rangle$} (1);
\path (1)   edge[in=-45,out=-15, loop, distance=2.5cm] node[right]  {} (1);
\node at (9,0) (dots) {$\vdots$ \hspace{0.5cm} $M$ loops};
\node at (7.6,1.4) (out) {$\alpha$};
\node at (8.5,0.5) (in) {$\beta$};
\end{tikzpicture}

Renaming the edge generators, we recognize this as an index $M$ subgroup of $G_{(M,0),(0,1)}$. Thus $G_{(1,0),(1,\ell)}$ is commensurable to $G_{(M,0),(0,1)}$. 

At this point, by Lemma \ref{basis} every linearly  $G_{(x,y),(z,w)}$ with nonzero intersection number is isomorphic to some $G_{(k,0),(j,\ell)}$, which in turn is commensurable to $G_{(M,0),(0,1)}$ for some $M\geq 1$.  We just saw that those are all commensurable to $G_{(1,0),(0,1)}$. Therefore the  $G_{(x,y),(z,w)}$ with nonzero intersection number are all commensurable to each other. 
\end{proof}

\subsection{Zero intersection number}

In the same way as before, we  first reduce to working with groups of the form $G_{(m,0),(n,0)}$. 

\begin{lemma}
A  tubular group $G_{(k,\ell),(m,n)}$ with $kn-\ell m=0$ is isomorphic to $G_{(\gcd(k,\ell),0),(\pm \gcd(m,n),0)}$. 
\end{lemma}

\begin{proof}
In the  case where $kn-\ell m=0$, we have that the indices are (nonzero) rational multiples of each other, i.e. $(m,n)=\frac{p}{q}(k,\ell)$ for $p,q\in \integers$ where $\frac{p}{q}$ is in lowest terms. So we have $(m,n)=\frac{p}{q}\gcd(k,\ell) (\frac{k}{\gcd(k,\ell)}, \frac{\ell}{\gcd(k,\ell)})$ where now $\frac{k}{\gcd(k,\ell)}$ and $\frac{k}{\gcd(k,\ell)}$ are relatively prime. 
Then $q \mid pk$ and $q\mid p\ell$ and $\gcd(p,q)=1$ which means $q\mid k$ and $q\mid \ell$, so $\gcd(k,\ell)=qr$ for some $r\in \integers$. We get that $(m,n)=pr (\frac{k}{\gcd(k,\ell)}, \frac{\ell}{\gcd(k,\ell)})$. But then note that $pr$ must equal $\gcd(m,n)$ up to sign. (We will see shortly that we can invert signs and get isomorphic groups, so the specific sign doesn't matter here,) Complete $a^{\frac{k}{\gcd(k,\ell)}}b^{\frac{\ell}{\gcd(k,\ell)}}$ to a $\integers^2$-basis $\{a^{\frac{k}{\gcd(k,\ell)}}b^{\frac{\ell}{\gcd(k,\ell)}}, a^M b^N\}$ in the same way as in the nonzero intersection number case. With respect to this new basis, we have $a^k b^\ell = (a^{\frac{k}{\gcd(k,\ell)}}b^{\frac{\ell}{\gcd(k,\ell)}})^{\gcd(k,\ell)}$ and $a^m b^n = (a^{\frac{k}{\gcd(k,\ell)}}b^{\frac{\ell}{\gcd(k,\ell)}})^{\gcd(m,n)}$. Thus $G_{(k,\ell),(m,n)} \simeq G_{(\gcd(k,\ell),0),(\pm \gcd(m,n),0)}$. 
     
\end{proof}

%\begin{itemize}
%    \item First of all, we can change basis and reduce to $G_{(m,0),(n,0)}$
%    \item Can show that $G_{(m,0),(n,0)}$ (WLOG for $\gcd(m,n)=1$) is commensurable to $G_{(km,0),(kn,0)}$ (for any $k\ne 0$). So we can focus on determining which $G_{(m,0),(n,0)}$ are commensurable, for $\gcd(m,n)=1$. 
%    \item Suspect that if $\frac{m}{n} \ne \frac{m'}{n'}$, then $G_{(m,0),(n,0)}$ won't be commensurable to $G_{(m',0),(n',0)}$, but I'm not sure yet
%\end{itemize}

$G_{(m,0),(n,0)}$ can be written as a graph of groups in different ways. We have the ``tubular group" graph of groups, with one $\integers^2$ vertex and one loop.  We could also decompose it as having one $BS(m,n)$ vertex and one loop. Moreover, $G_{(m,0),(n,0)}$ is actually a GBS group.

\begin{tikzpicture}
\node at (0,0) (1)[left]{$\langle  a, b \rangle$};
\node at (-0.9,0.05) (2)[left]{$\integers^2=$};
\path[every node/.style={font=\sffamily\small}]
(1)   edge[in=45,out=-45, loop, distance=2cm] node[right]  {$\langle a^m\rangle$} (1);
\node at (0.5,1) (in) {$a^m$};
\node at (0.5,-1) (out) {$a^n$};
\node at (0,-2) (label) {(Here the HNN letter is $t$)};

\node at (5,0) (3)[left]{$\langle  a, t \rangle$};
\node at (4.2,0) (4)[left]{$BS(m,n)=$};
\path[every node/.style={font=\sffamily\small}]
(3)   edge[in=45,out=-45, loop, distance=2cm] node[right]  {$\langle a^m\rangle$} (3);
\node at (5.5,1) (in) {$a$};
\node at (5.5,-1) (out) {$a$};
\node at (4.5,-2) (label) {(Here the HNN letter is $b$)};

\node at (10,0) (5)[left]{$\langle  a\rangle$};
\path[every node/.style={font=\sffamily\small}]
(5)   edge[in=45,out=-45, loop, distance=2cm] node[right]  {$\langle a^m\rangle$} (5);
\path[every node/.style={font=\sffamily\small}]
(5)   edge[in=135,out=-135, loop, distance=2cm] node[left]  {$\langle a\rangle$} (5);
\node at (11,1) (out) {$a^m$};
\node at (11,-1) (in) {$a^n$};
\node at (8.5,1) (out,l) {$a$};
\node at (8.5,-1) (in,l) {$a$};
\node at (9.5,-2)(label) {(HNN letters $b$ and $t$)};
\end{tikzpicture}

As in, $G_{(m,0),(n,0)}$ can be described by just a GBS labelled graph.

\begin{center}
\begin{tikzpicture}
    \node at (0,0) (5)[left]{};
    \draw[fill=black] (-0.1,0) circle (1.5pt);
\path[every node/.style={font=\sffamily\small}]
(5)   edge[in=45,out=-45, loop, distance=2.5cm] node[right]  {} (5);
\path[every node/.style={font=\sffamily\small}]
(5)   edge[in=135,out=-135, loop, distance=2.5cm] node[left]  {} (5);
\node at (1,1) (out) {$m$};
\node at (1,-1) (in) {$n$};
\node at (-1.5,1) (out,l) {$1$};
\node at (-1.5,-1) (in,l) {$1$};
\end{tikzpicture}
\end{center}

We can make a few more reductions by noticing that some of the $G_{(m,0),(n,0)}$'s are already commensurable. 
\begin{itemize}
    \item We always have $G_{(m,0),(n,0)}$ isomorphic to $G_{(n,0),(m,0)}$ (send the HNN letter of one to the inverse of the HNN letter of the other)
    \item We also have $G_{(m,0),(n,0)}$ isomorphic to $G_{(-m,0),(-n,0)}$ (if the relation is $ta^n t^{-1}=a^m$ in the first group and $s \widetilde{a}^{-n} s^{-1} = \widetilde{a}^{-m}$ in the second group, send $a$ to $\widetilde{a}^{-1}$)
    \item Furthermore, $G_{(m,0),(n,0)}$ is commensurable to $G_{(-m,0),(n,0)}$, as shown in the picture below.
\begin{center}
    \begin{tikzpicture}
    \node at (0,0) (1){};
    \draw[fill=black] (0,0) circle (1.5pt);
\path[every node/.style={font=\sffamily\small}]
(1)   edge[in=45,out=-45, loop, distance=2.5cm] node[right]  {} (1);
\path[every node/.style={font=\sffamily\small}]
(1)   edge[in=135,out=-135, loop, distance=2.5cm] node[left]  {} (1);
\node at (1,1) (out) {$-m$};
\node at (1,-1) (in) {$n$};
\node at (-1.25,1) (out,l) {$1$};
\node at (-1.25,-1) (in,l) {$1$};

\node at (0,1.5)(A){$\downarrow \: 2:1$};

    \node at (0,4) (2){};
    \draw[fill=black] (0,4) circle (1.5pt);
        \node at (0,6) (3){};
    \draw[fill=black] (0,6) circle (1.5pt);
\path[every node/.style={font=\sffamily\small}]
(2)   edge[in=-45,out=45,  distance=1cm] node[right]  {} (3);
\path[every node/.style={font=\sffamily\small}]
(2)   edge[in=-135,out=135, distance=1cm] node[left]  {} (3);
\path[every node/.style={font=\sffamily\small}]
(2)   edge[in=-135,out=-45, loop, distance=2.5cm] node[right]  {} (2);
\path[every node/.style={font=\sffamily\small}]
(3)   edge[in=135,out=45, loop, distance=2.5cm] node[right]  {} (3);
\node at (-1,5.5) (out) {$-m$};
\node at (-1,4.5) (out) {$n$};
\node at (1,5.5) (out) {$n$};
\node at (1,4.5) (out) {$-m$};
\node at (-1,7) (out,l) {$1$};
\node at (1,7) (out,l) {$1$};
\node at (-1,3) (in,l) {$1$};
\node at (1,3) (in,l) {$1$};

\node at (2,5) (B) {$\simeq$};

\node at (4,4) (4){};
    \draw[fill=black] (4,4) circle (1.5pt);
        \node at (4,6) (5){};
    \draw[fill=black] (4,6) circle (1.5pt);
\path[every node/.style={font=\sffamily\small}]
(4)   edge[in=-45,out=45,  distance=1cm] node[right]  {} (5);
\path[every node/.style={font=\sffamily\small}]
(4)   edge[in=-135,out=135, distance=1cm] node[left]  {} (5);
\path[every node/.style={font=\sffamily\small}]
(4)   edge[in=-135,out=-45, loop, distance=2.5cm] node[right]  {} (4);
\path[every node/.style={font=\sffamily\small}]
(5)   edge[in=135,out=45, loop, distance=2.5cm] node[right]  {} (5);
\node at (3,5.5) (out) {$m$};
\node at (3,4.5) (out) {$-n$};
\node at (5,5.5) (out) {$n$};
\node at (5,4.5) (out) {$-m$};
\node at (3,7) (out,l) {$1$};
\node at (5,7) (out,l) {$1$};
\node at (3,3) (in,l) {$1$};
\node at (5,3) (in,l) {$1$};

\node at (6,5) (B) {$\simeq$};

\node at (8,4) (6){};
    \draw[fill=black] (8,4) circle (1.5pt);
        \node at (8,6) (7){};
    \draw[fill=black] (8,6) circle (1.5pt);
\path[every node/.style={font=\sffamily\small}]
(6)   edge[in=-45,out=45,  distance=1cm] node[right]  {} (7);
\path[every node/.style={font=\sffamily\small}]
(6)   edge[in=-135,out=135, distance=1cm] node[left]  {} (7);
\path[every node/.style={font=\sffamily\small}]
(6)   edge[in=-135,out=-45, loop, distance=2.5cm] node[right]  {} (6);
\path[every node/.style={font=\sffamily\small}]
(7)   edge[in=135,out=45, loop, distance=2.5cm] node[right]  {} (7);
\node at (7,5.5) (out) {$m$};
\node at (7,4.5) (out) {$n$};
\node at (9,5.5) (out) {$n$};
\node at (9,4.5) (out) {$m$};
\node at (7,7) (out,l) {$1$};
\node at (9,7) (out,l) {$1$};
\node at (7,3) (in,l) {$-1$};
\node at (9,3) (in,l) {$-1$};

\node at (10,5) (B) {$\simeq$};

\node at (12,4) (8){};
    \draw[fill=black] (12,4) circle (1.5pt);
        \node at (12,6) (9){};
    \draw[fill=black] (12,6) circle (1.5pt);
\path[every node/.style={font=\sffamily\small}]
(8)   edge[in=-45,out=45,  distance=1cm] node[right]  {} (9);
\path[every node/.style={font=\sffamily\small}]
(8)   edge[in=-135,out=135, distance=1cm] node[left]  {} (9);
\path[every node/.style={font=\sffamily\small}]
(8)   edge[in=-135,out=-45, loop, distance=2.5cm] node[right]  {} (8);
\path[every node/.style={font=\sffamily\small}]
(9)   edge[in=135,out=45, loop, distance=2.5cm] node[right]  {} (9);
\node at (11,5.5) (out) {$m$};
\node at (11,4.5) (out) {$n$};
\node at (13,5.5) (out) {$n$};
\node at (13,4.5) (out) {$m$};
\node at (11,7) (out,l) {$1$};
\node at (13,7) (out,l) {$1$};
\node at (11,3) (in,l) {$1$};
\node at (13,3) (in,l) {$1$};

\node at (12,1.5)(B){$\downarrow \: 2:1$};

\node at (12,0) (10){};
    \draw[fill=black] (12,0) circle (1.5pt);
\path[every node/.style={font=\sffamily\small}]
(10)   edge[in=45,out=-45, loop, distance=2.5cm] node[right]  {} (10);
\path[every node/.style={font=\sffamily\small}]
(10)   edge[in=135,out=-135, loop, distance=2.5cm] node[left]  {} (10);
\node at (13,1) (out) {$m$};
\node at (13,-1) (in) {$n$};
\node at (11,1) (out,l) {$1$};
\node at (11,-1) (in,l) {$1$};
\end{tikzpicture}
\end{center}

In the top row, we have made \textit{admissible sign changes}, corresponding to choosing the other generator of vertex or edge groups. We first changed the sign of the left side non-loop edge labels, then the bottom vertex's labels, then finally the bottom loop's edges. These sign changes correspond to changing the choice of edge group or vertex group generator. 
    
\end{itemize}

\begin{lemma}
    $G_{(m,0),(n,0)}$ where $\gcd(m,n)=1$ is commensurable to $G_{(km,0),(kn,0)}$ for any nonzero integer $k$. 
\end{lemma}
\begin{proof}

We first handle the case where $m=n$. Consider the map $G_{(m,0),(m,0)} \twoheadrightarrow \integers/m\integers$ sending $a \mapsto 1, b\mapsto 0, t\mapsto 0$. The kernel of this map has graph of groups structure with one vertex and $|m|$ loops, depicted in the leftmost picture below. 

\begin{tikzpicture}
\node at (0,0) (1)[left]{$\langle  a^m, b \rangle$};
\path[every node/.style={font=\sffamily\small}]
(1)   edge[in=45,out=15, loop, distance=2.5cm] node[right]  {$\langle a^m \rangle$} (1);
\path (1)   edge[in=-45,out=-15, loop, distance=2.5cm] node[right]  {} (1);
\node at (2,0) (dots) {$\vdots \hspace{0.5cm} |m|$ loops};
\node at (0.6,1.4) (out) {$a^m$};
\node at (1.5,0.5) (in) {$a^m$};
\node at (4,0) (sim) {$\simeq$};

\node at (6,0) (1)[left]{$\langle  \alpha,\beta \rangle$};
\path[every node/.style={font=\sffamily\small}]
(1)   edge[in=45,out=15, loop, distance=2.5cm] node[right]  {$\langle \alpha \rangle$} (1);
\path (1)   edge[in=-45,out=-15, loop, distance=2.5cm] node[right]  {} (1);
\node at (8,0) (dots) {$\vdots$ \hspace{0.5cm} $|m|$ loops};
\node at (6.6,1.4) (out) {$\alpha$};
\node at (7.5,0.5) (in) {$\alpha$};

\node at (10.5,0) (sim) {$\simeq$};

\node at (12,0) (1)[left]{$\langle  a,b^m \rangle$};
\path[every node/.style={font=\sffamily\small}]
(1)   edge[in=45,out=15, loop, distance=2.5cm] node[right]  {$\langle a \rangle$} (1);
\path (1)   edge[in=-45,out=-15, loop, distance=2.5cm] node[right]  {} (1);
\node at (14,0) (dots) {$\vdots$ \hspace{0.5cm} $|m|$ loops};
\node at (12.6,1.4) (out) {$a$};
\node at (13.5,0.5) (in) {$a$};
\end{tikzpicture}

Letting $\alpha=a^m, \beta=b$, we get the middle picture. But then we recognize this as isomorphic to the rightmost picture, which is the kernel of the map $G_{(1,0),(1,0)} \twoheadrightarrow \integers/m\integers$ sending $a\mapsto 0,b \mapsto 1,t\mapsto 0$. So $G_{(m,0),(m,0)}$ is commensurable to $G_{(1,0),(1,0)}$ for any $m\ne 0$, and hence to each other.

Now suppose $m\ne n$. Consider the map $G_{(km,0),(kn,0)} \twoheadrightarrow \integers/k|m-n|\integers$ sending $a\mapsto 1,b\mapsto 0, t \mapsto 0$. The kernel of this map has graph of groups structure with one vertex and $\gcd(km,k|m-n|) = |k|\gcd(m,|m-n|) = |k|\gcd(m,n)=|k|$ many loops. Note $\lcm(km,k|m-n|) = |k m|\:|m-n|$. So the edge groups have generator $(a^{km})^{|m-n|}$ and the other map goes to $(a^{kn})^{|m-n|}$. 
%(the power of $a$ is the least common multiple up to sign)
\begin{center}
\begin{tikzpicture}
\node at (0,0) (1)[left]{$\langle  a^{k|m-n|}, b \rangle$};
\path[every node/.style={font=\sffamily\small}]
(1)   edge[in=45,out=15, loop, distance=2.5cm] node[right]  {$\langle a^{km|k-m|} \rangle$} (1);
\path (1)   edge[in=-45,out=-15, loop, distance=2.5cm] node[right]  {} (1);
\node at (2,0) (dots) {$\vdots \hspace{0.5cm} |k|$ loops};
\node at (0.6,1.4) (out) {$a^{km|m-n|}$};
\node at (2,0.5) (in) {$a^{kn|m-n|}$};
\node at (4,0) (sim) {$\simeq$};

\node at (6,0) (1)[left]{$\langle  \alpha,\beta \rangle$};
\path[every node/.style={font=\sffamily\small}]
(1)   edge[in=45,out=15, loop, distance=2.5cm] node[right]  {$\langle \alpha^m \rangle$} (1);
\path (1)   edge[in=-45,out=-15, loop, distance=2.5cm] node[right]  {} (1);
\node at (8,0) (dots) {$\vdots$ \hspace{0.5cm} $|k|$ loops};
\node at (6.6,1.4) (out) {$\alpha^m$};
\node at (7.5,0.5) (in) {$\alpha^n$};

\node at (10,0) (sim) {$\simeq$};

\node at (12,0) (1)[left]{$\langle  a,b^k \rangle$};
\path[every node/.style={font=\sffamily\small}]
(1)   edge[in=45,out=15, loop, distance=2.5cm] node[right]  {$\langle a^m \rangle$} (1);
\path (1)   edge[in=-45,out=-15, loop, distance=2.5cm] node[right]  {} (1);
\node at (14,0) (dots) {$\vdots$ \hspace{0.5cm} $|k|$ loops};
\node at (12.6,1.4) (out) {$a^m$};
\node at (13.5,0.5) (in) {$a^n$};
\end{tikzpicture}
\end{center}

Letting $\alpha=a^{k|m-n|},\beta=b$, we get the middle picture. But then we recognize that as a finite index subgroup of $G_{(m,0),(n,0)}$, namely the kernel of the map $G_{(m,0),(n,0)} \twoheadrightarrow \integers/k\integers$ sending $a \mapsto 0, b\mapsto 1, t\mapsto 0$. Therefore $G_{(km,0),(kn,0)}$ and $G_{(m,0),(n,0)}$ are commensurable. 
    
\end{proof}

Together with the last bullet point of the reductions from earlier, this means we can focus on determining which $G_{(m,0),(n,0)}$ are commensurable, for $\gcd(m,n)=1$ and $1\leq m \leq n$.

\begin{lemma}\label{Gm0n0_no_proper_plateay}
    For $\gcd(m,n)=1$, the labelled graph $BS(1,1) \vee BS(m,n)$ for $G_{(m,0),(n,0)}$ has no proper $p$-plateau for any prime $p$.
\end{lemma}

\begin{proof}
Let $A$ denote the labelled graph $BS(1,1) \vee BS(m,n)$. Suppose that $P \subs A$ is a $p$-plateau. Being a non-empty subgraph, $P$ must contain the vertex $v$, but note that $\{v\}$ by itself is not a plateau because no prime divides every label at $v$. So we would need at least one edge in $P$, along with $v$. 

\begin{center}
         \begin{tikzpicture}
    \node at (0,0) (v){$v$};
    \draw[->] (v) to[out=45,in=-45,distance=2.5cm] (v);
    \draw[->] (v) to[out=135,in=-135,distance=2.5cm] (v);
\node at (2,0) (out) {$e$};
\node at (-2,0) (in) {$f$};
\node at (1,1) (out) {$m$};
\node at (1,-1) (in) {$n$};
\node at (-1.5,1) (out,l) {$1$};
\node at (-1.5,-1) (in,l) {$1$};
\end{tikzpicture}
\end{center}

If $e\not\in E(P)$, then since $v\in V(P)$, it would mean that $p \mid m$ and $p\mid n$, giving a contradiction since $\gcd(m,n)=1$. So we must have $e\in E(P)$.  Then by definition of a plateau $p \not\mid m$ and $p\not\mid n$. But then $P$ would be a $p$-plateau for prime $p$ not dividing any label of $A$, so $P=A$. Therefore $A$ has no proper $p$-plateau for any prime $p$. This lets us use Theorem \ref{Lev15} later on. 
\end{proof}

\begin{lemma}\label{non-elementary}
    For $\gcd(m,n)=1$, the group $G_{(m,0),(n,0)}$ is non-elementary. 
\end{lemma}
\begin{proof}

Notice that $G_{(m,0),(n,0)} = \langle a,b,t: [a,b]=1, ta^n t^{-1} = a^m\rangle$ is non-abelian, because $b$ and $t$ do not commute. As such, $G_{(m,0),(n,0)}$ is not isomorphic to $\integers$ or $\integers^2$. To rule out the Klein bottle group $K$, we compare the rank of $G_{(m,0),(n,0)}$ and of $K = BS(1,-1)$. 

The Betti number of a graph $A$ is calculated by  $$\beta(A) = \frac{|E(A)|}{2} + |\text{connected components of }A| - |V(A)|$$ where we have divided by two to get number of edges in $A$ treated as an undirected graph. 

By the previous lemma, the only plateau for the labelled graph of $G_{(m,0),(n,0)}$ (where $\gcd(m,n)=1$) is the whole graph itself, which has one vertex, so $\mu(G_{(m,0),(n,0)})=1$. Meanwhile we have 2 edges, 1 connected component and 1 vertex, so $\beta(G_{(m,0),(n,0)}) =2$. Thus the rank of $G_{(m,0),(n,0)}$ is $3$. 

A very similar argument to the previous lemma shows that the labelled graph for $K=BS(1,-1)$ has no plateau either, and so $\mu(K)=1$ also. Meanwhile the labelled graph for $BS(1,-1)$ has 1 edge, 1 connected component and 1 vertex so $\beta(K) = 1$. Thus the rank of $K$ is $2$. This is different from the rank of $G_{(m,0),(n,0)}$, so $G_{(m,0),(n,0)}$ is not isomorphic to $K$. 
Therefore $G_{(m,0),(n,0)}$ is non-elementary. 
\end{proof}

\subsubsection{Using the modular homomorphism}

Since $G_{(m,0),(n,0)}$ for $\gcd(m,n)=1$ is non-elementary, we are able to define the modular homomorphism. Taking the loop perspective of the modular homomorphism, we see that the $(1,1)$ loop only contributes 1 to the image of $q$, so we just have $q(G_{(m,0),(n,0)}) = \langle \frac{m}{n} \rangle_{\rationals^*}$.  
\begin{center}
         \begin{tikzpicture}
    \node at (0,0) (v){};
  \draw[fill=black] (0,0) circle (1.5pt);
    \draw[->] (v) to[out=45,in=-45,distance=2.5cm] (v);
    \draw[->] (v) to[out=135,in=-135,distance=2.5cm] (v);
\node at (1,1) (out) {$m$};
\node at (1,-1) (in) {$n$};
\node at (-1.5,1) (out,l) {$1$};
\node at (-1.5,-1) (in,l) {$1$};
\end{tikzpicture}
\end{center}

By the Fact following Definition \ref{modularhom}, we get that if $G_{(m,0),(n,0)}$ were commensurable to $G_{(p,0),(q,0)}$, then we would have $\left(\frac{m}{n}\right)^k = \left(\frac{p}{q}\right)^\ell$ for some $k,\ell \in \integers\setminus \{0\}$. This already tells us that a lot of candidate pairs of groups are incommensurable, for example $G_{(1,0),(2,0)}$, $G_{(1,0),(3,0)}$ and $G_{(3,0),(5,0)}$ are pairwise incommensurable. 

The condition that $\left(\frac{m}{n}\right)^k = \left(\frac{p}{q}\right)^\ell$ for some $k,\ell \in \integers\setminus \{0\}$ requires  $\frac{m}{n}$ and $\frac{p}{q}$ to be powers of the same fraction $\frac{c}{d}$. We still need to investigate whether pairs like $G_{(1,0),(2,0)}$ and $G_{(1,0),(4,0)}$ are commensurable, or whether $G_{(2,0),(3,0)}$ and $G_{(4,0),(9,0)}$ are commensurable. But notice that the first two will have integers (other than $1$) in their image under the modular homomorphism, whereas the latter two will not, so the the first two are distinguishable from the latter two. 

As such, it's natural to split into two cases here: where the common power fraction has numerator 1, i.e.  groups of the form $G_{(1,0),(n^k,0)}$, and when it doesn't, i.e. groups of the form $G_{(m^k,0),(n^k,0)}$ for relatively prime $m,n\ne 1$. 

\subsubsection{Using the depth profile to distinguish \texorpdfstring{$G_{(1,0),(n^k,0)}$}{G}}

We are able to calculate the depth profile for $G_{(1,0),(N,0)}$ using Theorem \ref{Prop6.2For} (Proposition 6.2 in For24). The hypotheses there apply because $G_{(1,0),(N,0)} = BS(1,N) \vee BS(1,1)$. Let $V$ be the vertex group of $BS(1,N) \vee BS(1,1)$. We find that the depth profile is
$$\D(G_{(1,0),(N,0)},V) = \{N^i: i\in \naturals \cup \{0\}\}.$$
Recall that the depth profile itself isn't a commensurability invariant, but the equivalence class of a depth profile under the equivalence relation $S\subs \naturals$ is equivalent to $S/r := \{\frac{m}{\gcd(m,r)}\}$ (for any $r\in\naturals$) is a commensurability invariant. 

\begin{lemma}
    $G_{(1,0),(n^k,0)}$ is not commensurable to $G_{(1,0),(n^j,0)}$ for $j\ne k$. 
\end{lemma}
\begin{proof}
    For the group $G_{(1,0),(n^k,0)}$, denoting its vertex group as $V_k$, we have
$$\D(G_{(1,0),(n^k,0)},V_k) = \{n^{ki}: i\in\naturals \cup \{0\}\}.$$
Under the equivalence relation, it is equivalent to 
$$\D(G_{(1,0),(n^k,0)},V_k)/r = \left\{ \frac{n^{ki}}{\gcd(n^{ki},r)}: i\in \naturals \cup \{0\} \right\}$$ for any $r$. But note that for any $r$, if we order the elements by size, aside from finitely many exceptions any two successive elements have ratio $n^k$. It's because $\gcd(n^{ki},r)$ stabilizes as $i\to\infty$ to $r$ with all of its prime factors not dividing $n$ removed. So, if we look at ratios of successive elements, the tail of the sequence is invariant under the equivalence relation. For $G_{(1,0),(n^k,0)}$, the tail will be $n^k, n^k, n^k,\dots$. 

For $j\ne k$, we get a different tail for $\D(G_{(1,0),(n^j,0)},V_j)$ (where we are denoting the vertex group as $V_j$). The tail of the sequence of ratios of successive elements there will be $n^j, n^j, n^j,\dots$. As a result, $G_{(1,0),(n^k,0)}$ is not commensurable to $G_{(1,0),(n^j,0)}$ for $k\ne j$. 
\end{proof}

\begin{center}
         \begin{tikzpicture}
    \node at (0,0) (v){};
  \draw[fill=black] (0,0) circle (1.5pt);
    \draw[->] (v) to[out=45,in=-45,distance=2.5cm] (v);
    \draw[->] (v) to[out=135,in=-135,distance=2.5cm] (v);
\node at (1,1) (out) {$1$};
\node at (1,-1) (in) {$n^k$};
\node at (-1.5,1) (out,l) {$1$};
\node at (-1.5,-1) (in,l) {$1$};

    \node at (6,0) (u){};
  \draw[fill=black] (6,0) circle (1.5pt);
    \draw[->] (u) to[out=45,in=-45,distance=2.5cm] (u);
    \draw[->] (u) to[out=135,in=-135,distance=2.5cm] (u);
\node at (7,1) (out) {$1$};
\node at (7,-1) (in) {$n^j$};
\node at (4.5,1) (out,l) {$1$};
\node at (4.5,-1) (in,l) {$1$};
\end{tikzpicture}
\end{center}

This is surprising because from \cite{CRKZ} we know that $BS(1,n^k)$ is commensurable to $BS(1,n^j)$. Somehow by adding an extra loop to get $G_{(1,0),(n^k,0)}$ and $G_{(1,0),(n^j,0)}$, their depth profile's sequence of successive ratios will have different tails, which distinguishes them and makes them incommensurable now.

\subsubsection{Using slide moves to distinguish \texorpdfstring{$G_{(m^k,0),(n^k,0)}$}{G} for \texorpdfstring{$m,n\ne 1$}{not1}}

We saw in Lemma \ref{Gm0n0_no_proper_plateay} that the $BS(1,1) \vee BS(M,N)$ labelled graph for $G_{(M,0),(N,0)}$ for $\gcd(M,N)=1$ has no proper $p$-plateau for any prime $p$. Thus, by Theorem \ref{Lev15}, any finite index subgroup of $G_{(M,0),(N,0)}$ is actually a topological covering. 

\begin{center}
         \begin{tikzpicture}
         \node at (0,2.5)(B){$(B,\mu)$};
\node at (0,1.5) (arrow){$\downarrow$};
         
    \node at (0,0) (v){};
  \draw[fill=black] (0,0) circle (1.5pt);
    \draw[->] (v) to[out=45,in=-45,distance=2.5cm] (v);
    \draw[->] (v) to[out=135,in=-135,distance=2.5cm] (v);
\node at (1,1) (out) {$M$};
\node at (1,-1) (in) {$N$};
\node at (-1.5,1) (out,l) {$1$};
\node at (-1.5,-1) (in,l) {$1$};
\end{tikzpicture}
\end{center}

Every vertex in $(B,\mu)$ is (before performing any deformation moves) 4-valent and edges in $(B,\mu)$ either have labels $1,1$ or $M,N$. 

\begin{center}
\begin{tikzpicture}
     \node at (0,0) (v){};
       \draw[fill=black] (0,0) circle (1.5pt);
       \node at (2,0)(u){};
    \draw[fill=black] (2,0) circle (1.5pt);
    \draw[->] (v) to[out=0,in=180,distance=2.5cm] (u);
    \node at (0.5,0.25) (out) {$M$};
\node at (1.5,0.25) (in) {$N$};

\node at (3,0) (w){};
       \draw[fill=black] (3,0) circle (1.5pt);
    \draw[->] (w) to[out=45,in=-45,distance=2.5cm] (w);
        \node at (3.5,0.75) (out) {$M$};
\node at (3.5,-0.75) (in) {$N$};

     \node at (0,-2) (v'){};
       \draw[fill=black] (0,-2) circle (1.5pt);
       \node at (2,-2)(u'){};
    \draw[fill=black] (2,-2) circle (1.5pt);
    \draw[->] (v') to[out=0,in=180,distance=2.5cm] (u');
    \node at (0.5,-1.75) (out) {$1$};
\node at (1.5,-1.75) (in) {$1$};

\node at (3,-2) (w'){};
       \draw[fill=black] (3,-2) circle (1.5pt);
    \draw[->] (w') to[out=45,in=-45,distance=2.5cm] (w');
        \node at (3.5,-1.25) (out) {$1$};
\node at (3.5,-2.75) (in) {$1$};
\end{tikzpicture}
\end{center}

The edges could be loops, or they could be straight edges. But we can always get to a reduced labelled graph $(B,\widetilde{\mu})$ by performing collapse moves on all the $1,1$ straight edges. After doing collapses to get $(B,\widetilde{\mu})$, edges are now $M,N$ straight edges, $M,N$ loops or $1,1$ loops. (Since all the edges we collapsed were $1,1$ edges, the labels for the remaining edges were not affected.) $(B,\widetilde{\mu})$  represents the same GBS group as $(B,\mu)$, since it was obtained by elementary deformation moves. Notice that we could still potentially do slide moves to $(\widetilde{B}, \widetilde{\mu})$. Slide moves require divisibility between incident labels and the only divisibility we have are $1 \mid M$, $1 \mid N$, $M \mid M$ and $N \mid N$. Sliding with $1,1$ loops has no effect on labels. 

\begin{center}
\begin{tikzpicture}
     \node at (0,0) (1){};
       \draw[fill=black] (0,0) circle (1.5pt);
       \node at (2,0)(2){};
    \draw[fill=black] (2,0) circle (1.5pt);
        \draw[->] (1) to[out=0,in=180,distance=2.5cm] (2);
            \node at (1.5,-0.25) (out) {$N$};
\node at (0.5,-0.25) (in) {$M$};
\node at (-1,1) (3){};
\draw[fill=black] (-1,1) circle (1.5pt);
\draw[->] (1) to[out=135,in=-45,distance=0.5cm] (3);
            \node at (-0.5,0.25) (out) {$M$};
\node at (-1.2,0.8) (in) {$N$};

\node at (3,0)(arrow){$\xrightarrow[]{}$};

     \node at (4,0) (1){};
       \draw[fill=black] (4,0) circle (1.5pt);
       \node at (6,0)(2){};
    \draw[fill=black] (6,0) circle (1.5pt);
        \draw[->] (1) to[out=0,in=180,distance=2.5cm] (2);
            \node at (5.5,-0.25) (out) {$N$};
\node at (4.5,-0.25) (in) {$M$};
\node at (7,1) (3){};
\draw[fill=black] (7,1) circle (1.5pt);
\draw[->] (2) to[out=45,in=-135,distance=0.5cm] (3);
            \node at (6.5,0.25) (out) {$N$};
\node at (7.2,0.8) (in) {$N$};

\node at (0,-2) (1){};
       \draw[fill=black] (0,-2) circle (1.5pt);
       \node at (2,-2)(2){};
    \draw[fill=black] (2,-2) circle (1.5pt);
        \draw[->] (1) to[out=0,in=180,distance=2.5cm] (2);
            \node at (1.5,-2.25) (out) {$N$};
\node at (0.5,-2.25) (in) {$M$};
\draw[->] (2) to[out=45,in=-45,distance=2.5cm] (2);
\node at (2.5,-1.3) (out) {$M$};
\node at (2.5,-2.7) (out) {$N$};

\node at (4.5,-2)(arrow){$\xrightarrow[]{}$};

\node at (6,-2) (1){};
       \draw[fill=black] (6,-2) circle (1.5pt);
       \node at (8,-2)(2){};
    \draw[fill=black] (8,-2) circle (1.5pt);
        \draw[->] (1) to[out=0,in=180,distance=2.5cm] (2);
            \node at (7.5,-2.25) (out) {$M$};
\node at (6.5,-2.25) (in) {$M$};
\draw[->] (2) to[out=45,in=-45,distance=2.5cm] (2);
\node at (8.5,-1.3) (out) {$M$};
\node at (8.5,-2.7) (out) {$N$};

\end{tikzpicture}
\end{center}

The only thing that could happen is that an $M,N$ edge could slide over an incident $M,N$ edge to produce $M,M$ or $N,N$ edges (see picture for how these edges can be produced). But notice that the set of possible labels remains $\{1,M,N\}$ regardless of these slide moves. Also note that there are necessarily non-1 labels showing up. 

\begin{lemma}
    $G_{(m^i,0),(n^i,0)}$ is not commensurable to $G_{(m^j,0),(n^j,0)}$ for $i\ne j$. 
\end{lemma}
\begin{proof}
    Suppose for a contradiction that $G_{(m^i,0),(n^i,0)}$ is commensurable to $G_{(m^j,0),(n^j,0)}$ for $i\ne j$. Then  after doing collapses, we would have reduced labelled graphs $(\widetilde{B}, \widetilde{\mu})$ and $(\widehat{B}, \widehat{\mu})$, representing a finite index subgroup of $G_{(m^i,0),(n^i,0)}$ and of $G_{(m^j,0),(n^j,0)}$ respectively, where those finite index subgroups are isomorphic. The labelled graphs $(\widetilde{B}, \widetilde{\mu})$ and $(\widehat{B}, \widehat{\mu})$ represent the same GBS group $H$. But note that the image under the modular homomorphism of this GBS group is $q(H) = \langle \left(\frac{m^i}{n^i} \right)^\ell\rangle $ (where $\ell=[G_{(m^i,0),(n^i,0)}:H]$) and this has trivial intersection with $\integers$. Then by Theorem \ref{Prop3.5For}, $(\widetilde{B}, \widetilde{\mu})$ and $(\widehat{B}, \widehat{\mu})$ should be related by slide moves. However, there are no slide moves that can take one to the other. Even if we do all possible slides, labels in $(\widetilde{B}, \widetilde{\mu})$ can only ever be $\{1,m^i, n^i\}$ (where there are necessarily non-1 labels  appearing), whereas labels in $(\widehat{B}, \widehat{\mu})$ can only ever be $\{1,m^j,n^j\}$. Since $i\ne j$ and $\gcd(m,n)=1$, the non-1 labels are pairwise distinct. Therefore $G_{(m^i,0),(n^i,0)}$ and $G_{(m^j,0),(n^j,0)}$ cannot  be commensurable. 
\end{proof}

\section{Quasi-isometry classification}
\begin{lemma}
    $G_{(m,0),(n,0)}$ for $1 \leq m < n$ are quasi-isometric to $BS(2,3)$. Meanwhile, the groups $G_{(m,0),(m,0)}$ are in a separate quasi-isometry class. 
    %and are commensurable to $BS(n,n)$ for $n\geq 2$. 
\end{lemma}

\begin{proof}
We use Theorem \ref{Whyte} (Whyte's classification of GBS groups).  Starting with any $G_{(M,0),(N,0)}$ with $1 \leq M<N$, it is commensurable to (and so quasi-isometric to ) $G_{(m,0),(n,0)}$ where $\gcd(m,n)=1$ and $1\leq m<n$ still. We showed in Lemma \ref{non-elementary} that these are non-elementary. For non-elementary GBS groups, containing a finite index subgroup of the form $F_{k} \times \integers$ is equivalent to being unimodular. The $G_{(m,0),(n,0)}$ are not unimodular (their image under the modular homomorphism is $\langle \frac{m}{n} \rangle_{\rationals^*}$). Moreover, a rank argument shows that these $G_{(m,0),(n,0)}$ are not $BS(1,n)$. ($G_{(m,0),(n,0)}$ has rank 3, whereas $BS(1,n)$ has rank 2.) Therefore, by Theorem \ref{Whyte} , these $G_{(m,0),(n,0)}$ must be quasi-isometric to $BS(2,3)$.

Meanwhile, the groups $G_{(m,0),(m,0)}$ are commensurable to $G_{(1,0),(1,0)}$ which is non-elementary and unimodular, hence they have a finite index subgroup of the form $F_k \times \integers$. By Theorem \ref{Whyte}, they are not quasi-isometric to the $G_{(m,0),(n,0)}$ for $1 \leq m < n$.  
\end{proof}

\begin{lemma}
    $G_{(1,0),(0,1)}$ is not isomorphic to $G_{(1,0),(1,0)}$. 
\end{lemma}
\begin{proof}
First note that $a$ commutes with everything in $G_{(1,0),(1,0)} = \langle a,b,t: [a,b]=1, [a,t]=1\rangle$, so $G_{(1,0),(1,0)}$ has nontrivial center. We show that $G_{(1,0),(0,1)}$ has trivial center. 

Because $G_{(1,0),(0,1)}$ is an HNN extension of $BS(1,1)=\langle a,b :[a,b]=1\rangle$, we can write any $h\in G_{(1,0),(0,1)}$ as $$h=g_1 t^{\epsilon_1} g_2 t^{\epsilon_2}  \cdots g_n t^{\epsilon_n} g_{n+1}$$ where $\epsilon_i\in \{\pm 1\}$ and $g_i\in \langle a,b\rangle$. Britton's lemma says that if a word $h=1$ in $G_{(1,0),(0,1)}$, then either $n=0$ and $g_1=1$, or $n>0$ and for some $1\leq i \leq n-1$, one of the following holds:
\begin{itemize}
    \item $\epsilon_i = 1$ $\epsilon_{i+1}=-1$ and $g_i\in \langle b \rangle$
    \item $\epsilon_i=-1$, $\epsilon_{i+1}=1$ and $g_i \in \langle a \rangle$.
\end{itemize}

In fact, we can write $h$ in a normal form which is more restrictive than the previous form as $h=g_1 t^{\epsilon_1} g_2 t^{\epsilon_2} \cdots g_n t^{\epsilon_n} g_{n+1}$ where 
\begin{itemize}
    \item if $\epsilon_i =1$ then $g_i \in \{b^j:j\in \integers\}$ (since this is a transversal of $\langle a \rangle$)
    \item if $\epsilon_i=-1$ then $g_i \in \{a^j:j\in \integers\}$ (since this is a transversal of $\langle b \rangle$)
    \item If $\epsilon_{i-1}\ne \epsilon_i$ then $a_i \ne 1$
\end{itemize} and $g_{n+1}$ is arbitrary. Every element in the HNN extension is represented by a unique reduced word in this normal form. 

Suppose that $h\in G_{(1,0),(0,1)}$ commutes with $a$. Write $h=g_1 t^{\epsilon_1} g_2 t^{\epsilon_2}  \cdots g_n t^{\epsilon_n} g_{n+1}$ in the above normal form. Then $[h,a]=1$ means that
$$g_1 t^{\epsilon_1} g_2 t^{\epsilon_2}  \cdots g_n t^{\epsilon_n} g_{n+1} a g_{n+1}^{-1} t^{-\epsilon_n} g_n^{-1} \cdots t^{-\epsilon_2} g_2^{-1} t^{-\epsilon_1} g_1^{-1} a^{-1}=1$$
$$g_1 t^{\epsilon_1} g_2 t^{\epsilon_2}  \cdots g_n t^{\epsilon_n} a  t^{-\epsilon_n} g_n^{-1} \cdots t^{-\epsilon_2} g_2^{-1} t^{-\epsilon_1} g_1^{-1} a^{-1}=1$$
If $\epsilon_n=1$, then this word contains no subwords of the form $t b^{m} t^{-1}$ or $t^{-1} a^{m} t$ and Britton's lemma would say that the word cannot equal $1$. So actually we must have $\epsilon_{n}=-1$. Then the subword gives $t^{-1} a t = b$ so we have
$$g_1 t^{\epsilon_1} g_2 t^{\epsilon_2}  \cdots g_{n-1} t^{\epsilon_{n-1}} g_n b g_n^{-1} t^{-\epsilon_{n-1}} g_{n-1}^{-1} \cdots t^{-\epsilon_2} g_2^{-1} t^{-\epsilon_1} g_1^{-1} a^{-1}=1$$
$$g_1 t^{\epsilon_1} g_2 t^{\epsilon_2}  \cdots g_{n-1} t^{\epsilon_{n-1}} b  t^{-\epsilon_{n-1}} g_{n-1}^{-1} \cdots t^{-\epsilon_2} g_2^{-1} t^{-\epsilon_1} g_1^{-1} a^{-1}=1$$ and here by the same reasoning, we must have $\epsilon_{n-1}=1$, so that cancellation can occur. Continuing in this way, we see that the powers have to alternate sign for us to keep cancelling the word. If $n$ is odd, the last cancellation we do produces a power of $b$ and we get $g_1 b g_1^{-1} a^{-1}=1$ and $g_1\in \langle a,b\rangle$ will commute with $b$, giving $b a^{-1}=1$ which is a contradiction. Meanwhile if $n$ is even, the last cancellation is $t b t^{-1}=a$ and we get $g_1 a g_1^{-1} a^{-1}$ which is fine. So if $[h,a]=1$, it must be of the form $h=g_1 t g_2 t^{-1} \cdots t^{-1} g_{2m+1}$, as in it has an even number of appearances of $t$ (in either sign) and their powers alternate, starting with $1$ and ending with $-1$.

By a similar argument, we get that if $[h,b]=1$, it must be of the form $h=g_1 t^{-1} g_2 t \cdots t g_{2m+1}$, i.e. it has an even number of appearances of $t$ and their powers alternate, starting with $-1$ and ending with $1$. But reduced words are unique and $h$ has both of these reduced word forms. Comparing them, we get that $h$ must have zero appearances of $t$, since the two forms disagree on what their powers should be. So  $h=g_1 \in \langle a,b\rangle$ (it is the rightmost letter, which is allowed to be anything in $\integers^2 = \langle a,b:[a,b]=1\rangle$). Write $h=a^k b^\ell$.

Now suppose that $[t, a^k b^\ell]=1$, i.e. 
$$t a^k b^\ell t^{-1} a^{-k} b^{-\ell}=1.$$ We can replace $b^\ell t^{-1}$ with $t^{-1}a^\ell$ giving
$$t a^k t^{-1} a^\ell  a^{-k} b^{-\ell}=1$$
which is now a nontrivial reduced word (so it can't equal 1) unless $k=0$. But if $k=0$, we have $a^\ell b^{-\ell} =1$ which is only possible $\ell = 0$. But then $k=0$ and $\ell=0$ means that $a^kb^\ell = 1$. So actually the only element of $\langle a,b:[a,b]=1\rangle$ that commutes with $t$ is $1$. Then since the center $Z(G_{(1,0),(0,1)})$ is the intersection of all the centralizers of elements (taken in $G_{(1,0),(0,1)}$), we have $$Z(G_{(1,0),(0,1)}) \subs C_{G_{(1,0),(0,1)}}(a) \cap C_{G_{(1,0),(0,1)}}(b) \cap C_{G_{(1,0),(0,1)}}(t) = 1$$
since we just saw that if $h\in G_{(1,0),(0,1)}$ commutes with both $a$ and $b$, it will only commute with $t$ if it's trivial. Thus $G_{(1,0),(0,1)}$ has trivial center. Then it cannot be isomorphic to $G_{(1,0),(1,0)}$. 
\end{proof}

\begin{lemma}\label{G_10,01 JSJ}
    The following graph of groups decomposition of $G_{(1,0),(0,1)}$ is a JSJ decomposition. 
\begin{center}
    \begin{tikzpicture}
    \node at (0,0) (5)[left]{};
    \draw[fill=black] (0,0) circle (1.5pt);
    \node at (-1.5,0) (out) {$\langle a,b:[a,b]=1\rangle$};
\path[every node/.style={font=\sffamily\small}]
(5)   edge[in=45,out=-45, loop, distance=2.5cm] node[right]  {} (5);
\path[every node/.style={font=\sffamily\small}];
\node at (1,1) (out) {$a$};
\node at (1,-1) (in) {$b$};
\end{tikzpicture}
\end{center}
\end{lemma}

\begin{proof}
We check the conditions in Definition \ref{JSJCashen}. $G_{(1,0),(0,1)}$ is certainly one-ended group and we argue that it is not commensurable to a surface group. A finite index subgroup of a surface group is another surface group. Because $G_{(1,0),(0,1)}$ is torsion-free (HNN extensions preserve torsion free-ness), it cannot be the fundamental group of a non-orientable surface. It also cannot be isomorphic to $\pi_1(S_g)$ for $g\geq 2$, where $S_g$ is the closed orientable surface of genus $g$, because $G_{(1,0),(0,1)}$ is not hyperbolic, and it cannot be $\pi_1(S_1) = \integers^2$ because it's not abelian. 

The decomposition has one vertex with vertex group $G_v = \integers^2$, which is not two-ended. It is not hanging, because $\integers^2$ is not hyperbolic and so cannot be quasi-isometric to the universal cover of the hyperbolic pair of pants. We show that the $\integers^2$ vertex group is rigid. This will be all we need, because the last two conditions in Definition \ref{JSJCashen} are vacuously true: there are no valence one vertices in the graph of groups and there are no hanging vertices in the graph of groups, so no hanging vertices the Bass-Serre tree either. 

Suppose that the $G_v = \integers^2$ vertex group was not rigid. Then it splits over a two-ended subgroup  in such a way that the incident edge groups $\langle a\rangle$ and $\langle b \rangle$ are conjugate into one of the vertex groups in the supposed splitting of $G_v$. A priori, the splitting of $G_v$ could be into any graph of groups $\Gamma$, but we can reduce to the case where $\Gamma$ has a single edge, either a loop (if the vertex groups that $\langle a \rangle$ and $\langle b \rangle$ conjugate into are part of a cycle) or a straight edge. So either $\integers^2 = G_v = H *_C K$ or $H *_C$ where $C$ is a two-ended group. 

\begin{center}
\begin{tikzpicture}
     \node at (0,0) (1){};
       \draw[fill=black] (0,0) circle (1.5pt);
       \node at (2,0)(2){};
    \draw[fill=black] (2,0) circle (1.5pt);
        \draw[-] (1) to[out=0,in=180,distance=2.5cm] (2);
        \node at (0,0.5) (out) {$H$};
        \node at (1,0.25) (out) {$C$};
        \node at (2,0.5) (out) {$K$};
        \node at (-0.5, -1) (a) {$\langle a \rangle$};
        \draw [{Hooks[right]}->] (a) -- (1);
        \node at (2.5, -1) (b) {$\langle b \rangle$};
        \draw [{Hooks[left]}->] (b) -- (2);

     \node at (4,0) (3){};
       \draw[fill=black] (4,0) circle (1.5pt);
       \node at (6,0)(4){};
    \draw[fill=black] (6,0) circle (1.5pt);
        \draw[-] (3) to[out=0,in=180,distance=2.5cm] (4);
          \node at (4,0.5) (out) {$H$};
        \node at (5,0.25) (out) {$C$};
        \node at (6,0.5) (out) {$K$};  
                \node at (7, 0.75) (A) {$\langle a \rangle$};
        \draw [{Hooks[left]}->] (A) -- (4);
        \node at (7, -0.75) (B) {$\langle b \rangle$};
        \draw [{Hooks[left]}->] (B) -- (4);
            
     \node at (2,-3) (5){};
       \draw[fill=black] (2,-3) circle (1.5pt);
       \draw[-] (5) to[out=135,in=225,distance=2.5cm] (5);
                       \node at (3, -2.25) (a') {$\langle a \rangle$};
        \draw [{Hooks[left]}->] (a') -- (5);
        \node at (3, -3.75) (b') {$\langle b \rangle$};
        \draw [{Hooks[left]}->] (b') -- (5);
                  \node at (2,-2.5) (out) {$H$};
        \node at (0.25,-3) (out) {$C$};
\end{tikzpicture}
\end{center}

In the first two cases where $G_v = H *_C K$, if $C$ has index 2 or more in both $H$ and $K$, then there will be elements in the amalgamated product that do not commute, contradicting how $\integers^2$ is abelian. Then $C$ must be index $1$ in at least one of $H$ or $K$, so it is the trivial splitting which can be collapsed to the other vertex group. 

Note that because $C$ is a two-ended subgroup of $G_v = \integers^2$, it is torsion free and finitely generated (it is a subgroup of a finitely generated abelian group). Then $C$ is a finitely generated two-ended group so it must be virtually cyclic, and in fact cyclic because it is torsion free.  We have that $gag^{-1}, hbh^{-1} \in C$ for some $g, h \in G_{(1,0),(0,1)}$. %Let $C = \langle a^i b^j \rangle$.

Write $g$ in reduced form as $g = g_1 t^{\epsilon_1} \cdots g_n t^{\epsilon_n} g_{n+1}$. For $g a g^{-1}$ to equal an element $a^k b^\ell \in C$ means that 
$$g_1 t^{\epsilon_1} \cdots g_n t^{\epsilon_n} g_{n+1} a g_{n+1}^{-1} t^{-\epsilon_n} g_n^{-1} \cdots t^{-\epsilon_1} g_1^{-1} a^{-k} b^{-\ell} = 1 $$
$$g_1 t^{\epsilon_1} \cdots g_n t^{\epsilon_n} a  t^{-\epsilon_n} g_n^{-1} \cdots t^{-\epsilon_1} g_1^{-1} a^{-k} b^{-\ell} = 1 $$
If $\epsilon_n=1$, then this word contains no subwords of the form $tb^m t^{-1}$ or $t^{-1} a^m t$ and Britton's lemma would say that the word cannot equal $1$. So actually we must have $\epsilon_n = -1$. Then the subword gives $t^{-1} a t = b$ so we have
$$g_1 t^{\epsilon_1} \cdots g_{n-1} t^{\epsilon_{n-1}} g_n b g_n^{-1}  t^{-\epsilon_{n-1}} g_{n-1}^{-1}\cdots t^{-\epsilon_1} g_1^{-1} a^{-k} b^{-\ell} = 1 $$
$$g_1 t^{\epsilon_1} \cdots g_{n-1} t^{\epsilon_{n-1}} b  t^{-\epsilon_{n-1}} g_{n-1}^{-1}\cdots t^{-\epsilon_1} g_1^{-1} a^{-k} b^{-\ell} = 1 $$
and here by the same reasoning we must have $\epsilon_{n-1}=1$, so that cancellation can occur. Continuing in this way, we see that the powers have to alternate sign for us to keep cancelling the word. If $n$ is even, the cancellation gives $a a^{-k} a^\ell = 1$, so $k=1, \ell=0$ and we have $a=a^k b^\ell = gag^{-1}$. If $n$ is odd, the cancellation gives $b a^{-k} b^{-\ell} = 1$ so $k=0, \ell=1$ and $gag^{-1}=b$. We see that conjugating $a$ by $g\in G_{(1,0), (0,1)}$ can only ever produce $a$ or $b$. Similarly, conjugating $b$ by $h\in G_{(1,0),(0,1)}$ can only ever produce $a$ or $b$. So the vertex group $H$ contains at least one of $a$ and $b$. Suppose that $a\in H$.

If $H$ contains any element with a nonzero power of $b$, then because the edge group is cyclic, the HNN letter will not commute with both $a$ and that element with a nonzero power of $b$, and so we do not get an abelian group like $\integers^2$. 
\begin{center}
\begin{tikzpicture}
    \node at (2,-3) (5){};
       \draw[fill=black] (2,-3) circle (1.5pt);
       \draw[-] (5) to[out=135,in=225,distance=2.5cm] (5);
                  \node at (3,-3) (out) {$H \sups \langle a \rangle$};
        \node at (-0.25,-3) (out) {$\langle a^i b^j \rangle$};
\end{tikzpicture}
\end{center}

So suppose that $H$ does not contain any element with a nonzero power of $b$, so that $H = \langle a \rangle$. Then the edge group, being a subgroup of $H$, is generated by a power of $a$. In fact, it must be $\langle a \rangle$ mapped by the identity on both ends, otherwise the HNN letter in the supposed $G_v$ will not commute with $a$. The incident edge groups $\langle a \rangle$ and $\langle b \rangle$ must both conjugate to $a$ in order to have image inside of $H$. But then the resulting HNN extension gives $G_{(1,0), (1,0)}$ which is not isomorphic to $G_{(1,0),(0,1)}$. Therefore, there is no possible nontrivial splitting of the vertex group $G_v$ over two-ended edge groups and  $G_v$ is rigid. 
\end{proof}

\begin{lemma}
    $G_{(1,0),(0,1)}$ is not quasi-isometric to $G_{(m,0),(n,0)}$ for any $m,n$. 
\end{lemma}

\begin{proof}
We argue that the JSJ tree of cylinders for $G_{(m,0),(n,0)}$ is just one point, whereas that of $G_{(1,0),(0,1)}$ is not. Then because the two groups have non-isomorphic trees of cylinders, they cannot be quasi-isometric. 

The $BS(1,1)\vee BS(m,n)$ graph of groups decomposition for $G_{(m,0),(n,0}$ is a JSJ decomposition. We can check this using Definition \ref{JSJCashen}, which is applicable because $G_{(m,0),(n,0)}$ is one-ended and not commensurable to a surface group (by the same reasoning as for $G_{(1,0),(0,1)}$ previously). The only vertex group is $\integers$ which is 2-ended. There are no valence 1 vertices and by Lemma 2.11 in \cite{ForDeform}, GBS groups have no hanging vertices. 

Because this is a GBS graph of groups for $G_{(m,0),(n,0)}$, edge stabilizers in the corresponding Bass-Serre tree have finite index in neighboring vertex stabilizers and the Bass-Serre tree is locally finite. Then all vertex and edge stabilizers are commensurable (in the sense of their intersection being finite index in both groups). Thus there is just one cylinder (all the edges are in one cylinder) and there are no vertices in two or more cylinders, so the JSJ tree of cylinders is just one point. 

Meanwhile, we work with the JSJ decomposition of $G_{(1,0),(0,1)}$ from Lemma \ref{G_10,01 JSJ} and its corresponding Bass-Serre tree. Here, cylinders are all of the outgoing edges starting at some vertex $h$ and going to either $ha^i t$ (if the net sum of powers of $t$ is odd) or $hb^i t$ (if the net sum is even). So each cylinder consists of the infinitely many edges that share the same origin vertex. The stabilizer of the cylinder with origin vertex $h$ is $\langle h a h^{-1} \rangle$. These are incommensurable for distinct $h$ in the Bass-Serre tree. So there are infinitely many cylinders in the Bass-Serre tree for $G_{(1,0),(0,1)}$ which already give rise to infinitely many vertices in the JSJ tree of cylinders. So it cannot be isomorphic to the JSJ tree of cylinders for $G_{(m,0),(n,0)}$. 

\end{proof}

\begin{center}
\begin{tikzpicture}
    \draw (0,-1) -- (0,5) -- (9,5) -- (9,-1) -- (0,-1);
    \draw (4.5,5) -- (4.5,-1);
    \draw (0,4.2) -- (4.5,4.2);
    \draw (0,3.4) -- (4.5,3.4);
    \draw (0,2.6) -- (4.5,2.6);
    \draw (0,1.8) -- (4.5,1.8);
     \node at (2.2,5.5) (LD) {Zero intersection number};
    \node at (6.7,5.5) (LI) {Nonzero intersection number};
    \node at (2.2,4.6) () {$G_{(1,0),(1,0)} \sim G_{(m,0),(m,0)}$};
    \node at (2.2,3.8) () {$G_{(1,0),(2,0)} \sim G_{(2,0),(4,0)}$ etc.};    
    \node at (2.2,3) () {$G_{(1,0),(4,0)} \sim G_{(2,0),(8,0)}$ etc.}; 
    \node at (2.2,2.2) () {$G_{(2,0),(3,0)} \sim G_{(4,0),(6,0)}$ etc.}; 
    \node at (2.2,1) () {$\vdots$}; 
    \node at (6.7,3.8) () {All are commensurable}; 
    \node at (6.7,3.2) () {to $G_{(1,0),(0,1)}$}; 
    \node at (6.7,1.8) () {$G_{(2,1),(0,3)},$}; 
    \node at (6.7,1) () {$G_{(0,1),(-4,6)}$ etc.}; 
     \draw [blue,thick]  (-0.05,-1) -- (-0.05,4.2) -- (4.55,4.2) -- (4.55,-1.05) -- (-0.05,-1.05);
     \draw [blue,thick] (-0.05,5.05) -- (4.55,5.05) -- (4.55,4.2) -- (-0.05,4.2) -- (-0.05,5.05);
     \draw [blue, thick] (4.55,5.05) -- (9.05,5.05) -- (9.05,-1.05) -- (4.55, -1.05) -- (4.55,5.05);
     \node at (4.5,-2) () {Black boxes are commensurability classes};
     \node at (4.5,-2.6) () {\textcolor{blue}{Blue} boxes are quasi-isometry classes};
     \node at (-1,2.2) () {\textcolor{blue}{$BS(2,3) \sim$}};
\end{tikzpicture}
\end{center}

\printbibliography
\end{document}